\documentclass[11pt]{amsart}
\usepackage{fullpage,verbatim,amssymb}
\usepackage{hyperref}
\usepackage[active]{srcltx}
\usepackage[cmtip,arrow,matrix,curve,tips,frame]{xy}
\usepackage[usenames,dvipsnames]{color}

\makeatletter
\let\@@pmod\mod
\DeclareRobustCommand{\mod}{\@ifstar\@pmods\@@pmod}
\def\@pmods#1{\mkern4mu({\operator@font mod}\mkern 6mu#1)}
\makeatother

\definecolor{blue}{rgb}{0,0,1}
\definecolor{red}{rgb}{1,0,0}
\definecolor{green}{rgb}{0,.6,.2}
\definecolor{purple}{rgb}{1,0,1}
\long\def\red#1\endred{\textcolor{red}{#1}}
\long\def\blue#1\endblue{\textcolor{blue}{#1}}
\long\def\purple#1\endpurple{\textcolor{purple}{ #1}}
\long\def\green#1\endgreen{\textcolor{green}{#1}}

\newcommand{\bsl}{\backslash}

\newcommand{\Z}{\mathbb{Z}}
\newcommand{\Q}{\mathbb{Q}}
\newcommand{\R}{\mathbb{R}}
\newcommand{\C}{\mathbb{C}}

\newcommand{\HH}{\mathbb{H}}

\newcommand{\scrL}{\mathcal{L}}

\newcommand{\cuspa}{\mathfrak{a}}

\DeclareMathOperator{\SL}{SL}

\DeclareMathOperator{\GL}{GL}

\DeclareMathOperator{\ord}{{ord}}

\DeclareMathOperator{\lcm}{{lcm}}
\DeclareMathOperator{\Sym}{{Sym}}

\newcommand{\sgn}{{\rm sign}} 

\newcommand{\cont}{{\rm cont}}

\newcommand{\Res}{{\rm Res}}
\newcommand{\sym}{{\rm sym}}

\newcommand{\sm}{\left(\begin{smallmatrix}}
\newcommand{\esm}{\end{smallmatrix}\right)}
\newcommand{\bpm}{\begin{pmatrix}}
\newcommand{\ebpm}{\end{pmatrix}}

\newtheorem{theorem}{Theorem}
\newtheorem{lemma}[theorem]{Lemma}
\newtheorem{proposition}[theorem]{Proposition}
\newtheorem{corollary}[theorem]{Corollary}

\theoremstyle{remark}

\newtheorem{remark}[theorem]{Remark}

\numberwithin{theorem}{section}
\numberwithin{equation}{section}

\title{First moments of Rankin-Selberg convolutions of automorphic forms on $\GL(2)$}
\author{Jeff Hoffstein, Min Lee and Maria Nastasescu}
\thanks{M.~L.\ was supported by Royal Society University Research Fellowship ``Automorphic forms,
L-functions and trace formulas''. (AMS classifications 11M32, 11M36)}
\address{Mathematics Department, Brown University, Providence RI 02912, USA }
\email{\tt jhoff@math.brown.edu}
\address{School of Mathematics, University of Bristol, Bristol, BS8 1TW, UK}
\email{\tt min.lee@bristol.ac.uk}
\address{Mathematics Department, Northwestern University, Evanston, IL 60208, USA}
\email{\tt mnastase@math.northwestern.edu}

\begin{document}

\maketitle

\begin{abstract}
We obtain a first moment formula for Rankin-Selberg convolution $L$-series 
of holomorphic modular forms or Maass forms of arbitrary level on $\GL(2)$, 
with an orthonormal basis of Maass forms.    
One consequence is the best result to date, uniform in level, spectral value and weight, 
for the equality of two Maass or holomorphic cusp forms if their
Rankin-Selberg convolutions with the orthonormal basis of Maass forms $u_j$ is equal at the center 
of the critical strip for sufficiently many $u_j$.   

The main novelty of our approach is the new way the error
terms are treated.  
They are brought into an exact form that
provides optimal estimates for this first moment case, and also provide a basis
for an extension to second moments, which will appear in another work. 
\end{abstract}

\tableofcontents

\section{Introduction}\label{s:intro}
The objective of this paper is to obtain a first moment formula for Rankin-Selberg convolution $L$-series 
of holomorphic modular forms or Maass forms of arbitrary level on $\GL(2)$, 
with an orthonormal basis of Maass forms, and to derive some consequences.   
A second paper  \cite{HL20s}  builds on this result to obtain a second moment formula, with additional applications.  

Let $\phi$ be an automorphic form of level $N$.
The spectral  first moment  we intend to analyze is:
\begin{equation}\label{e:firstmoment_prel}
\sum_j \frac{h(r_j)}{\cosh(\pi r_j)} \rho_j(n) \scrL(s, \phi\times \overline{u_j}) 
+ \text{ continuous spectrum, }
\end{equation}
where 
the $\{u_j\}_{j\geq 1}$ are an orthonormal basis for Maass forms of level $M$ (divisible by $N$) 
with Laplace eigenvalues $\frac{1}{4}+r_j^2$, and $\rho_j(n)$ their $n$th Fourier coefficients. 
Here $\scrL(s, \phi\times \overline{u_j})$ is the Rankin-Selberg convolution of $\phi$ with $u_j$ as given in \eqref{e:RS_phiuj}.
The $u_j$ could be an oldform or a newform. 
We are able to treat the Rankin-Selberg $L$-functions of newforms and oldforms in a uniform manner as we are not using the approximate functional equation. 
We only use the functional equation of the additive twist $L$-functions of $\phi$, which is a newform. 
We describe our methods in the next paragraph. 
The continuous spectrum part is given in \eqref{e:K_firstmoment_def}.
The function $h(r)$ is a test function that limits the range of  the parameter $r=r_j$ for the Laplace eigenvalues, 
which satisfies the conditions given in \S\ref{ss:first}. 
In particular, for our applications, we choose $h(r)$ such that it decays exponentially when $|T-|r|| \gg T^{\alpha}$ and remains close to constant when $|T-|r|| \ll T^{\alpha}$ for a fixed $T\gg 1$ and $\frac{1}{3} < \alpha< \frac{2}{3}$, 
as given in \eqref{e:testh_f}.
 
Our approach will be to start with the left hand side of the first moment formula \eqref{e:firstmoment_prel}. 
For sufficiently large $\Re(s)>1$, we will open up the Rankin-Selberg convolutions, 
and apply the Bruggeman-Kuznetsov trace formula \eqref{e:Kuznetsov}. 
This process leads us to certain $L$-series of $\phi$ twisted by Kloosterman sums and $J$-Bessel functions  
\eqref{e:K0_open}.
These $L$-series are studied in \S\ref{ss:L-K-B}. 
To accomplish this analysis we use the Mellin inversion of the $J$-Bessel function, and rewrite the series as 
the inverse Mellin transform of a linear combination of additive twists of $L$-functions of $\phi$.    
We then use the meromorphic continuation of additive twists of the $L$-function of $\phi$, 
and apply its functional equation in Proposition~\ref{prop:L_fe}. 
Then we study the Mellin-type integrals \eqref{e:F1_def} and \eqref{e:F2_def}, 
in Lemma~\ref{lem:F1} and Lemma~\ref{lem:F2}, respectively. 
When $2u$ is an integer ( here $u$ is a parameter in the Mellin-type integrals in Lemma~\ref{lem:F2};  $2u\in \Z$ when $u_j$'s are holomorphic cusp forms), 
this integral was studied in \cite{Byk96} and \cite{GZ99}. 
More recently \cite{Nel13} also used the results and method of  \cite{GZ99}. 
A similar approach has been used in \cite{HR19}.

This analysis leads to an exact expression for the first moment in terms of a main term plus precise error terms.   
These are given in Theorem~\ref{thm:first_nonholo} and Theorem~\ref{thm:upperbound_first}.   
The error terms are  related to  shifted Dirichlet series, 
which are bounded in Theorem~\ref{thm:upperbound_first}.  
The precise description of these error terms as  inverse Mellin transforms of  shifted Dirichlet series 
will be vital for our followup paper  \cite{HL20s} analyzing second moments of Rankin-Selberg convolutions.
The specific case where $\phi$ is an Eisenstein series leads to a second moment 
for the standard $L$-series of the Masss forms, 
which is described in Corollary~\ref{cor:first_eis}.  
Finally, in Corollary~\ref{cor:determine} we derive a consequence, uniform in level, spectral value and weight, 
showing that if two forms $\phi$, $\phi'$ have 
$\scrL(1/2, \phi\times \overline{u_j})= \scrL(1/2, \phi'\times \overline{u_j})$ for sufficiently many (but finitely many)
$j$, then $\phi = \phi'$.

The most recent work on first moments of Rankin-Selberg $L$-series that we are aware of is \cite{H20}, 
where for a fixed $g$ of level $M$, an estimate is made for the sum of $L(1/2,f \times g)$ with the sum over holomorphic cusp forms $f$ of fixed prime power level, relatively prime to $M$. 
This is used to find a subconvexity upper bound in the level aspect for the central value $L(1/2,f\times g)$ in the case that all the central values are non-negative, which occurs when $g$ is dihedral.  
This strengthens the results of a previous work \cite{HT14}

\begin{remark}
We would like to stress here that, 
the key to obtaining the precise expression for the spectral first moment 
as a main term plus error terms (which are shifted Dirichlet series) 
is the transformation of integrals done in \S\ref{ss:F12}. 
Goldfeld and Zhang \cite{GZ99} studied the first moment of the Rankin-Selberg convolution 
averaging over the holomorphic cusp forms with the same method. 

In addition to allowing optimal upper bound estimates, this precise description allows us,  in a following paper \cite{HL20s},f to take a sum over $n$ and build up a second moment formula from the first moment.   
When this is done, one of the error terms from the first moment builds a piece of the main term of the second moment.
\end{remark}

\begin{remark}
We assume the divisibility of the levels for simplicity. 
To remove this condition, one needs to consider other types of additive twists of the $L$-functions of $\phi$. 
For example, see \cite[Theorem~3.1]{DHKL}. 
\end{remark}

\par\vspace{2\jot}\noindent
\textbf{Acknowledgements}.\ 
The authors would like to thank Peter Humphries for some 
very helpful comments, and POSTECH for providing a welcoming working
environment during part of the preparation of this paper.

\subsection{Preliminaries}\label{ss:notation}
Let $\HH=\{x+iy\;:\; x\in \R, \; y>0\}$ be the Poincar\'e upper half plane 
and $\GL_2^+(\R)$ be the group of $2\times2$ matrices over $\R$ with positive determinants. 
The group $\GL_2^+(\R)$ acts on $\HH$ via the M\"obius transformation: 
for $g=\sm a & b\\ c & d\esm \in \GL_2^+(\R)$ and $z\in \HH$ we define $gz = \frac{az+b}{cz+d}$.
For a given integer $k$, the group $\GL_2^+(\R)$ acts on a function $\phi: \HH\to \C$ via the slash operator of weight $k$:
\begin{equation}
\phi|_k g (z) = \det(g)^{\frac{k}{2}} (cz+d)^{-k} \phi(gz). 
\end{equation}
When $k=0$ we write $\phi|\gamma (z)= \phi|_0 \gamma (z) = \phi(gz)$. 

Let $N$ be a positive integer and $\Gamma_0(N)$ be the congruence subgroup of $\SL_2(\Z)$: 
\begin{equation}
\Gamma_0(N) = \left\{\gamma=\bpm a & b\\ c& d\ebpm \in \SL_2(\Z)\;:\; c\equiv 0\bmod{N}\right\}. 
\end{equation}
Let $\phi$ be either a holomorphic cusp form of positive integral weight $k$, or a weight $0$ Maass form 
(cuspidal or Eisenstein series) of type $\nu\in \C$, 
for $\Gamma_0(N)$, with central character $\xi$ modulo $N$, i.e., 
$\phi$ is a real analytic function on $\HH$ satisfying
\begin{align}
&(\phi|_k\gamma)(z) = \xi(d) \phi(z) & \text{ for any }\gamma=\sm a & b\\ c& d\esm\in \Gamma_0(N), \\
&\Delta \phi = (1/4-\nu^2) \phi & \text{ when $\phi$ is a Maass form.}
\end{align}
Here $\Delta = -y^2\left(\frac{\partial^2}{\partial x^2} + \frac{\partial^2}{\partial y^2}\right)$ is the Laplace-Beltrami operator 
(with the negative sign). 
We further let $\nu=\frac{k-1}{2}$ when $\phi$ is a holomorphic cusp form of weight $k$. 

We now notate the Fourier expansion of $\phi$. 
When $\phi$ is a weight $k$ holomorphic cusp form, 
\begin{equation}\label{e:Fourierexp_holo}
\phi(z) = \sum_{n=1}^\infty c_\phi(n) e^{2\pi inz}
= \sum_{n=1}^\infty C_\phi(n)n^{\frac{k-1}{2}}  e^{2\pi inz}. 
\end{equation}
When $\phi$ is non-holomorphic, for simplicity we consider only weight $0$ automorphic forms 
 with the following Fourier expansion: 
\begin{equation}\label{e:Fourierexp_nonholo}
\phi(z) 
= c_\phi^+ y^{\frac{1}{2}+\nu} + c_\phi^- y^{\frac{1}{2}-\nu} 
+ \sum_{n=1}^\infty C_\phi(n) \sqrt{y} K_\nu(2\pi ny) e^{2\pi inx}
+ c_\phi(-1) \sum_{n=1}^\infty C_\phi(n) \sqrt{y} K_{\nu}(2\pi ny) e^{-2\pi in x}. 
\end{equation}
Here $K_\nu(y)$ is the classical $K$-Bessel function 
\begin{equation}
K_\nu(y) = \frac{1}{2}\int_0^\infty e^{-\frac{1}{2} y\left(u+u^{-1}\right)} u^{\nu} \; \frac{du}{u}. 
\end{equation}
We further assume that $\phi$ is a newform for level $N$, and is a Hecke eigenform.

\subsection{Spectral decomposition for $L^2(\Gamma_0(M)\bsl \HH)$}\label{ss:spectral_decomp}
Let $M$ be a positive integer. 
For any $\varphi_1, \varphi_2\in L^2(\Gamma_0(M) \bsl \HH)$, we define the Petersson inner product 
 for $\Gamma_0(M)\bsl \HH$: 
\begin{equation}
\left<\varphi_1, \varphi_2\right>_M = \iint_{\Gamma_0(M)\bsl \HH} \varphi_1(z)\overline{\varphi_2(z)} \; \frac{dx\; dy}{y^2}.
\end{equation}
The space $L^2(\Gamma_0(M)\bsl \HH)$, with respect to the Petersson inner product, 
can be decomposed into the eigenspaces  of the Laplace-Beltrami operator $\Delta$. 

Let $\{u_j\}_{j\geq 1}$ be an orthonormal basis of the cuspidal part of the space $L^2(\Gamma_0(M)\bsl \HH)$, 
which are Maass cusp forms with Laplace eigenvalue $s_j(1-s_j)$ for $s_j\in \C$.
Note that either $s_j=\frac{1}{2}+ir_j$ for $r_j\in \R$ or $\frac{1}{2}< s_j < 1$. 
We have the following Fourier expansion 
\begin{equation}
u_j(z) = \sum_{n\neq 0} \rho_j(n) \sqrt{y} K_{ir_j}(2\pi|n|y) e^{2\pi inx}. 
\end{equation}
We will assume that the $u_j$'s have been further diagonalized with respect to Hecke operators $T_n$ 
for $n\in \Z_{\geq 1}$, $\gcd(n, M)=1$,  
and let $\lambda_j(n)$ be the corresponding eigenvalue for $u_j$.
 When $u_j$ is a newform of level $M$, we have $\rho_j(\pm n)=\rho_j(\pm 1)\lambda_j(n)$ for $n\geq 1$. 


The Eisenstein series for $\Gamma=\Gamma_0(M)$ are indexed by the cusps $\cuspa\in \Q\cup\{\infty\}$.
For each cusp $\cuspa\in \Q\cup\{\infty\}$, let $\sigma_\cuspa\in \SL_2(\R)$, 
with $\sigma_\cuspa \infty = \cuspa$, 
be a scaling matrix for the cusp $\cuspa$, i.e., 
$\sigma_\cuspa$ is the unique matrix (up to right translations)
such that 
$\sigma_\cuspa \infty=\cuspa$
and 
\begin{equation}
\sigma_\cuspa^{-1}\Gamma_\cuspa \sigma_\cuspa 
= \Gamma_\infty = \left\{\left. \pm \bpm 1 & b\\ 0 & 1\ebpm \;\right|\; b\in \Z\right\},
\end{equation}
where
\begin{equation}
\Gamma_\cuspa 
= \left\{\gamma\in \Gamma\;|\; \gamma \cuspa = \cuspa\right\}.
\end{equation}
For a cusp $\cuspa$, define the Eisenstein series at the cusp $\cuspa$ to be
\begin{equation}\label{e:Ea}
E_\cuspa(z, s) 
=
\sum_{\gamma\in \Gamma_\cuspa \bsl \Gamma} \Im(\sigma_\cuspa^{-1}\gamma z)^s, 
\end{equation}
with the following Fourier expansion:
\begin{equation}\label{m:EisensteinSeries_fex}
E_\cuspa \left(z, s\right) 
= \delta_{\cuspa, \infty} y^{s}
+ \tau_\cuspa \left(s, 0\right) y^{1-s}
+ \sum_{n\neq 0} \tau_\cuspa \left(s, n\right) \sqrt y K_{s-\frac{1}{2}}(2\pi |n|y)e^{2\pi inx}.
\end{equation}
For explicit descriptions of the Eisenstein series $E_{\cuspa}(z, s)$ see \cite[Theorem~6.1]{Y19}.

\subsection{First moments}\label{ss:first}
We define a test function $h(t)$ satisfying the following conditions:
\begin{enumerate}
\item $h(t)$ is even;
\item $h(t)$ is holomorphic in the strip $|\Im(t)|\leq 1/2 +\epsilon$;
\item $h(t)\ll (|t|+1)^{-2-\epsilon}$ in the strip;
\item $h(\pm i/2)=0$. 
\end{enumerate}

For a positive integer $N$, let $\phi$ be an automorphic form of level $N$ with central character $\xi$ modulo $N$ as given in \S\ref{ss:notation}. 
We further assume that $\xi$ is induced from a primitive character $\xi_*$ with conductor $R\mid N$. 
For a positive integer $M$ divisible by $N$, we consider the following Rankin-Selberg convolutions, $\phi$ with 
the spectral basis of $L^2(\Gamma_0(M)\bsl \HH)$ given in \S\ref{ss:spectral_decomp}. 
For $\Re(s)>1$, 
\begin{equation}\label{e:RS_phiuj}
\scrL(s, \phi\times u_j)
= L^{(M)}(2s, \xi) \sum_{m=1}^\infty \frac{C_\phi(m)\rho_j(m)}{m^s}
\end{equation}
and 
\begin{equation}\label{e:RS_phiE}
\scrL_\cuspa (s, ir; \phi)
= L^{(M)}(2s, \xi) \sum_{m=1}^\infty \frac{C_\phi(m)\overline{\tau_\cuspa(1/2+ir, m)}}{m^s}. 
\end{equation}
Here 
\begin{equation}
L(2s, \xi)= \sum_{m=1}^\infty \frac{\xi(m)}{m^{2s}} = \prod_{p} (1-\xi(p)p^{-2s})^{-1}
\end{equation}
and 
\begin{equation}
L^{(M)}(2s, \xi) = L(2s, \xi\cdot 1_M) = L(2s, \xi) \prod_{p\mid M} (1-\xi(p)p^{-2s})
\end{equation}
where $1_M$ is the trivial character mod $M$. 

For each positive integer $n$, we define the first moment function
\begin{multline}\label{e:K_firstmoment_def}
K(s, \phi; n, h)
\\ = \sum_j \frac{h(r_j)}{\cosh(\pi r_j)} \rho_j(n) \scrL(s, \phi\times \overline{u_j}) 
+ \sum_{\cuspa} \frac{1}{4\pi} \int_{-\infty}^\infty \frac{h(r)}{\cosh(\pi r)} 
\tau_\cuspa(1/2+ir, n) \scrL_\cuspa(s, ir; \phi) \; dr.
\end{multline}
In Theorem~\ref{thm:first_holo} and Theorem~\ref{thm:first_nonholo}, 
we present the spectral first moment of Rankin-Selberg convolutions 
as separated ``main terms", some extra terms and Mellin inversions of the following shifted Dirichlet series. 
For $s\in \C$ and $m\in \Z$, let 
\begin{equation}\label{e:PM_xi*}
P_{M}(u, m; \xi_*) 
= \prod_{\substack{p\mid M, \\ p\nmid R}} 
\bigg\{-p^{-1-u}\xi_*(p)+ (1-p^{-1}) \sigma_u(p^{\ord_p(m)-\ord_p(M)}; \overline{\xi_*})\bigg\}, 
\end{equation}
\begin{equation}
\sigma_{1-2s}(m; \chi) = \sum_{d\mid m} d^{1-2s}\chi(d), 
\end{equation}
for a Dirichlet character $\chi$ 
and define 
\begin{equation}\label{e:sigma_xi*_M}
\sigma_{1-2s}(m; \xi_*, M)
= 
\sigma_{1-2s}(m; \xi_* \cdot 1_M) P_M(2s-1, m, \xi_*) 
\prod_{p\nmid M}\overline{\xi_*(p)}^{\ord_p(m)} \prod_{p\mid M} p^{(1-2s)\ord_p(m)} 
\end{equation}
for a non-zero integer $m$, which is divisible by $\frac{M}{R\prod_{p\mid M, p\nmid R} p}$. 
When $\frac{M}{R\prod_{p\mid M, p\nmid R} p}\nmid m$, then we define $\sigma_{1-2s}(m; \xi_*, M)=0$. 
 In particular, when $\xi$ is primitive then $\xi_*=\xi$, $R=M$ and for a non-zero integer $m$ with $\gcd(m, R)=1$, 
we have $P_{R}(2s-1, m; \xi_*)=1$ and 
\begin{equation}
\sigma_{1-2s}(m; \xi_*, M)
= \sigma_{1-2s}(m; \xi_*) \overline{\xi_*(m)}. 
\end{equation}
This $\sigma_{1-2s}(m; \xi_*, M)$ is the $m$th Fourier coefficient of an Eisenstein series of level $M$ 
and character $\xi$, at the cusp  $\frac{1}{M}$, which is equivalent to $\infty$ via $\Gamma_0(M)$ (see \cite{Y19}). 

We then define the following shifted Dirichlet series of $\phi$ and the Eisenstein series of level $M$ with the character $\xi_*$. 
For a positive integer $n$, 
\begin{equation}\label{e:Dphi}
D(s, it; \phi; n) = \sum_{m=1}^\infty \frac{C_\phi(m+n)(m+n)^{\nu} \sigma_{-2it}(m; \xi_*, M)m^{it}}{m^{s+\nu}}.
\end{equation}
Here we take $\nu=\frac{k-1}{2}$ when $\phi$ is a holomorphic cusp form of weight $k$, 
and when $\phi$ is non-holomorphic, $\frac{1}{4}-\nu^2$ is the Laplace eigenvalue for $\phi$. 
The series converges absolutely for $\Re(s)>1$. 
This is a generalization of a shifted Dirichlet series that was first defined and meromorphically continued in  \cite{HHR},
for two holomorphic forms of weight $k$.  
If $n$ were negative this would be a type of shifted Dirichlet series first defined by Selberg in \cite{Sel}, 
but as $n$ is positive it has very different analytic  properties. 

Now we describe $K(s, \phi; n, h)$ for $\Re(s)=\frac{1}{2}$ in the following two theorems. 
We let 
\begin{equation}\label{e:H0}
H_0(h) = \frac{1}{\pi^2} \int_{-\infty}^\infty h(r) r\tanh(\pi r) \; dr
\end{equation}
and 
\begin{equation}\label{e:Hk-1/2_h}
H_{\frac{k-1}{2}} (s; h) = \frac{1}{\pi^2} \int_{-\infty}^\infty h(r) r\tanh(\pi r) 
\frac{\Gamma\left(1-s+\frac{k-1}{2}+ir\right) \Gamma\left(1-s+\frac{k-1}{2}-ir\right)}
{\Gamma\left(s+\frac{k-1}{2}+ir\right)\Gamma\left(s+\frac{k-1}{2}-ir\right)}\; dr.
\end{equation}


\begin{theorem}\label{thm:first_holo}
Let $\phi$ be a holomorphic cusp form of weight $k$. 
For $\Re(s)\geq1/2$, we have 
\begin{equation}
K(s, \phi; n, h)
= M(s, \phi; n) + L_{1}^+(s, \phi; n) + L_1^-(s, \phi; n), 
\end{equation}
where 
\begin{multline}\label{e:M_holo}
M(s, \phi; n)
= L^{(M)}(2s, \xi) \frac{C_\phi(n)}{n^s} H_0(h)
\\ + \delta_{\xi=1_N} (2\pi)^{-2+4s}\zeta(2-2s) 
M^{1-2s} \prod_{p\mid M} (1-p^{-1}) \frac{C_\phi(n)}{n^{1-s}} 
H_{\frac{k-1}{2}}(s; h), 
\end{multline}
\begin{multline}\label{L1+}
L_{1}^+(s, \phi; n)
= i^k (2\pi)^{2s-1} R^{-2s} \tau(\xi_*)
\frac{4}{2\pi i} \int_{(\sigma_u)}\frac{h(u/i) u}{\cos(\pi u)} 
\frac{\sin\left(\pi\left(s-u+\frac{k-1}{2}\right)\right)}{\pi} 
\frac{\Gamma\left(1-s-\frac{k-1}{2}+u\right)}{\Gamma\left(s+ \frac{k-1}{2}+u\right)}
\\ \times 
\frac{1}{2\pi i } 
\int_{(\sigma_v)} \frac{\Gamma\left(1-s+ \frac{k-1}{2}+v\right) \Gamma\left(s+ \frac{k-1}{2}+v\right) \Gamma\left(u-v\right)}{\Gamma\left(1+u+v\right)}
n^v D(v+1/2, s-1/2; \phi; n)
\; dv \; du, 
\end{multline}
where $\Re(s) < \sigma_v< \sigma_u< 3/2$, 
and 
\begin{multline}\label{L1-}
L_1^-(s, \phi; n) 
= i^k (2\pi)^{2s-1} R^{-2s} \tau(\xi_*)
\frac{\cos(\pi(s-\frac{k-1}{2}))}{\pi} 
\frac{4}{2\pi i}\int_{C} 
\frac{h(u/i)u \tan(\pi u)}{\Gamma\left(s-u+\frac{k-1}{2}\right)\Gamma\left(s+u+\frac{k-1}{2}\right)}
\\ \times 
\frac{1}{2\pi i} \int_{(\sigma_v)} \Gamma\left(u-v\right) \Gamma\left(-u-v\right) 
\Gamma\left(s+\frac{k-1}{2}+v\right) \Gamma\left(1-s+\frac{k-1}{2}+v\right)
\\ \times n^v \sum_{m=1}^{n-1} \frac{C_\phi(n-m)(n-m)^{\frac{k-1}{2}} \sigma_{1-2s}(m; \xi_*, M)}{m^{1-s+v+\frac{k-1}{2}}} 
\; dv \; du, 
\end{multline}
where $0< \sigma_u< \frac{1}{2}$ and $C$ is the contour that separates the poles of the gamma functions. 
\end{theorem}

When $\phi$ is not holomorphic, we have a similar formula, with extra terms $M^\pm$ and $L_2$, 
as described in the following theorem.    
When $\nu\notin \frac{1}{2} \Z$, define 
\begin{equation}\label{e:Hnu_h}
H_\nu(s; h) 
= \frac{4}{\pi} \frac{1}{2\pi i} \int_{(\sigma_u)}h(u/i) u
\frac{\cos(\pi(s-u))}{\cos(\pi u)} 
\frac{\Gamma\left(1-s-\nu+u\right) \Gamma\left(1-s+\nu+u\right)}
{\Gamma\left(s+\nu+u\right)\Gamma\left(s-\nu+u\right)}
\; du, 
\end{equation}
for $\frac{1}{2}\leq \Re(s) < \sigma_u  < \frac{3}{2}$. 
We remark that if $\phi$ were an Eisenstein series there would be an extra term contributed from the continuous spectrum after analytically continuing back to $\Re(s)=1/2$.   
We will give this extra term explicitly in Corollary~\ref{cor:first_eis} 
(when $\phi$ is taken as the Eisenstein series of level $1$), 
even though it turns out to be on the same order as the error term.
\begin{theorem}\label{thm:first_nonholo}
Let $\phi$ be a Maass form of type $\nu$. 
For $\frac{1}{2}\leq \Re(s)< \frac{3}{2}$, we have 
\begin{multline}
K(s, \phi; n, h)
= M(s, \phi; n) + M^+(s, \phi; n) + M^-(s, \phi; n)
\\ + L_{1}^+(s, \phi; n) + L_1^-(s, \phi; n) + L_2(s, \phi; n), 
\end{multline}
where 
\begin{multline}\label{e:Mdef_nonholo}
M(s, \phi; n)
= L^{(M)}(2s, \xi) \frac{C_\phi(n)}{n^s} H_0(h)
\\ -\delta_{\xi=1_N}\zeta(2-2s) (2\pi)^{4s-3} 
M^{1-2s} \prod_{p\mid M}(1-p^{-1}) \frac{C_\phi(n)}{n^{1-s}} \frac{\pi }{\sin(\pi s) }H_\nu(s; h).
\end{multline}
Here $M^\pm(s, \phi; n)$, $L_1^\pm(s, \phi; n)$ and $L_2(s, \phi; n)$ 
are given in \eqref{e:M+}, \eqref{e:L1+nonholo_holo}, \eqref{e:L1-_withF1} and \eqref{e:L2} respectively.  
Note that when $\phi$ is cuspidal, $M^\pm(s, \phi; n) =0$.  
\end{theorem}

\begin{remark}
The function $L_1^+(s, \phi; n)$ and $L_2(s, \phi; n)$ are Mellin inversions of shifted Dirichlet series \eqref{e:Dphi} 
of $\phi$ and an Eisenstein series of level $M$ associated with a Dirichlet character $\xi$. 
For $L_2$ the role of $\phi$ and the Eisenstein series is switched. 
The functions $L_1^- (s, \phi; n)$ is a Mellin inversion of a short shifted sum. 
As mentioned above, these shifted sums, although an error term in the first moment, 
become a vital part of the analysis in a further application of this first moment formula to a second moment \cite{HL20s}.
In particular, after a spectral decomposition, the sum over $n$ of $L_1^+(s, \phi; n)$ becomes part of the main term in the second moment, 
although the sum over $n$ of $L_1^- (s, \phi; n)$ remains part of the error term.
\end{remark}

\begin{remark}
When $\xi$ is the trivial character modulo $N$, since $\frac{1}{\sin(\pi s)}$ has a pole at $s=1$, 
the term $M(s, \phi; n)$ also has a pole at $s=1$ with the residue 
\begin{equation}
\Res_{s=1} M(s, \phi; n)
= 2\pi \frac{\prod_{p\mid M}(1-p^{-1})}{M} C_\phi(n) h(\nu/i). 
\end{equation}
As expected from the spectral side: when $\phi$ is a cuspidal newform of level $N$, 
since $N\mid M$, there exists $j_0$ such that $\nu=ir_{j_0}$, induced from a newform of level $N$ which is a constant multiple of $\phi$. 
Then the Rankin-Selberg convolution $\scrL(s, \phi\times \overline{u_{j_0}})$ has a pole at $s=1$, 
so the spectral first moment function $K(s, \phi; n, h)$ 
has a pole at $s=1$ with the residue
\begin{equation}
\Res_{s=1} K(s, \phi; n, h)
= \frac{h(r_{j_0})}{\cosh(\pi r_{j_0})} \rho_{j_0}(n) \Res_{s=1}\scrL(s, \phi\times \overline{u_{j_0}}). 
\end{equation}
These poles must cancel, meaning that the residues at the pole must be equal.  
This is not obvious at first. 
For example, when $N=M$,  then $u_{j_0} = \rho_{j_0}(1) \phi$ and 
$|\rho_{j_0}(1)|^2$ is equal to $(L^*(1, \sym^2 u_{j_0}))^{-1}$, the inverse of the completed symmetric square $L$-series, multiplied by a constant depending on the level.   
Thus the two symmetric square $L$-series in the numerator and denominator cancel, and the $\cosh(\pi r_{j_0})$ in the denominator cancels the inverse of the gamma functions in the completed symmetric square.  
The equality of the remaining constant can be checked.
When $N\mid M$ and $N<M$, then the relation is similar but a little bit more complicated.  
\end{remark}

By choosing $\phi(z) = \frac{1}{2} E^*(z, 1/2+it)$, the complete Eisenstein series of level $1$, 
(see \eqref{e:Eis_1}), for $t\in \R$
we obtain an explicit formula for the spectral second moments of $L$-functions of Maass cusp forms for $\SL_2(\Z)$, 
as given in the corollary below. 
Our method works for arbitrary $N$, and it is possible to work out a similar formula in full generality, 
but in this corollary we restrict ourselves to level $1$ for simplicity. 
\begin{corollary}\label{cor:first_eis}
Take $\phi(z) = \frac{1}{2} E^*(z, 1/2+it)$.
For any positive integer $n$, 
\begin{multline}\label{e:first_eis_level1}
\sum_j \frac{h(r_j)}{\cosh(\pi r_j)} |\rho_j(1)|^2 \lambda_j(n) 
L(1/2+it, u_j) L(1/2-it, u_j) 
\\ + \frac{1}{4\pi } \int_{-\infty}^\infty \frac{h(r)}{\cosh(\pi r)} 
\frac{4\zeta(1/2+it+ir)\zeta(1/2+it-ir)\zeta(1/2-it+ir)\zeta(1/2-it-ir)}{\zeta^*(1+2ir)\zeta^*(1-2ir)}
\sigma_{-2ir}(n)n^{ir} \; dr
\\= M(1/2, \phi; n) + M^+(1/2, \phi; n) + M^-(1/2, \phi; n)
-2\bigg\{h((1/2+ it)/i)\frac{\zeta(1+2it)}{\zeta(2+2it)}
+h((1/2-it)/i)\frac{\zeta(1-2it)}{\zeta(2-2it)}\bigg\}
\\ + L_{1}^+(1/2, \phi; n) + L_1^-(1/2, \phi; n) + L_2(1/2, \phi; n), 
\end{multline}
where 
\begin{multline}\label{e:M_Eis_s=1/2}
M(1/2, \phi; n) 
= \frac{\sigma_{-2it}(n)n^{it}}{\sqrt{n}} 
\frac{1}{\pi^2} \int_{-\infty}^\infty h(r) r \tanh(\pi r) 
\\ \times 
\frac{1}{2} \bigg(\psi\left(\frac{1}{2}+it+ir\right) + \psi\left(\frac{1}{2}+it-ir\right)
+ \psi\left(\frac{1}{2}-it+ir\right) + \psi\left(\frac{1}{2}-it-ir\right)\bigg)
\; dr
\\ + \frac{\sigma_{-2it}(n)n^{it}}{\sqrt{n}} \big(-\log((2\pi)^2 n) + 2\gamma_0 \big) H_0(h)
\end{multline}
and 
\begin{multline}\label{e:Mpm_s=1/2}
M^\pm (1/2, \phi; n)
= 2\frac{\zeta(1\pm 2it) (2\pi)^{-1\mp 2it}}{\Gamma\left(1\pm it\right)\Gamma\left(\mp it\right)}
\sigma_{0}(n) n^{-\frac{1}{2}\mp it} 
\\ \times \frac{1}{\pi^2} \int_{-\infty}^\infty h(r)r \tanh(\pi r) 
\Gamma\left(\frac{1}{2}\pm it+ir\right) \Gamma\left(\frac{1}{2}\pm it-ir\right)\; dr.
\end{multline}
Here $\gamma_0$ is the Euler-Mascheroni constant and $\psi(x) = \frac{\Gamma'}{\Gamma}(x)$ is the digamma function. 
The Dirichlet series associated with $L_1^+$ is 
\begin{equation}
D(1/2+v, 0; \phi; n) = \sum_{m=1}^\infty \frac{\sigma_{-2it}(m+n)(m+n)^{it} \sigma_0(m)}{m^{\frac{1}{2}+v+it}}. 
\end{equation}
\end{corollary}

\begin{remark}
By choosing the test function $h=h_{T, \alpha}$ as in \eqref{e:testh_f} below, 
the error terms satisfy the following bound:
\begin{equation}
L_{1}^+(1/2, \phi; n) + L_1^-(1/2, \phi; n) + L_2(1/2, \phi; n)
= O_{\epsilon}(n^{\frac{1}{2}+\theta+\epsilon} T^{1+\epsilon}), 
\end{equation}
as $T\to \infty$. 
Here we assume that $|t| \ll T^{1-\epsilon}$ for some $\epsilon>0$. 
For the asymptotic behaviour of the main term $M(1/2, \phi; n)$ and $M^\pm(1/2, \phi; n)$ see Theorem~\ref{thm:upperbound_first} below. 
Also, the extra term from the continuous spectrum, moved to the right hand side of \eqref{e:first_eis_level1}, 
is absorbed in the error term.
\end{remark}

For our applications, we take the following test function:   
Fix some $C \gg 1$, 
$0< \alpha \leq 1$ and $T \gg 1$, and set
\begin{equation}\label{e:testh_f}
h(r) = h_{T, \alpha} (r) = \left(e^{-\left(\frac{r-T}{T^\alpha} \right)^2}+ e^{-\left(\frac{r+T}{T^\alpha} \right)^2}\right)
\frac{r^2+\frac{1}{4}}{r^2+C}. 
\end{equation}
Then we have 
\begin{theorem}\label{thm:upperbound_first}
We follow the same set-up as in Theorem~\ref{thm:first_holo} and Theorem~\ref{thm:first_nonholo}
and assume that $\phi$ is either a holomorphic cusp form or a Maass cusp form with the trivial character modulo $N$. 

In a formula where $k$ appears and $\phi$ is non-holomorphic, we will set $k = 0$. 
For any positive integer $n$, we have the following formula for the first moment of the Rankin-Selberg convolution:
\begin{multline}\label{e:first_est}
\sum_{j} \frac{h_{T, \alpha} (r_j)}{\cosh(\pi r_j)} \overline{\rho_j(n)} \scrL(1/2+it, \phi\times u_j) 
+ \sum_{\cuspa} \frac{1}{4\pi} \int_{-\infty}^\infty \frac{h_{T, \alpha} (r)}{\cosh(\pi r)} 
\tau_\cuspa(1/2+ir, n) \scrL_\cuspa(1/2+it, ir; \phi) \; dr
\\ = M(1/2+it, \phi; n) 
+ \begin{cases}
O_{\epsilon} (M^{\epsilon}2^{\frac{k}{2}} n^{\frac{1}{2}+\epsilon} T^{\max\{1, \alpha+\beta\}+\epsilon}) & \text{ if } \phi \text{ is holomorphic}, \\
O_{\epsilon} (M^{\epsilon} n{^{\frac12+\theta+\epsilon}} T^{\max\{1, \alpha+\beta\}+\epsilon})& \text{ if }  \phi \text{  is non-holomorphic},
\end{cases}
\end{multline}
where $\theta$ denotes the best progress toward the Ramanujan-Petersson conjecture 
and we write $|t| = T^{\beta}$, with $ \beta <1$, taking $\beta=-\infty$ when $t=0$. 

The error may be written more precisely as 
\begin{equation}
M^\epsilon n^\epsilon \left(T^{\alpha +\beta +\epsilon}\sum_{m=1}^{n-1} \frac{|C_\phi(n-m)|}{m^{1/2+\epsilon}}
+T^{1+\epsilon} 2^{\frac{k}{2}} \sum_{m=1}^{n} \frac{|C_\phi(n+m)|}{m^{1/2+\epsilon}}
+T^{-A}  n \sum_{m>n} \frac{|C_\phi(n+m)|}{m^{3/2+\epsilon}} \right)
\end{equation}
for arbitrarily large $A$. 
The implied constant in the estimate is independent of $M$, $n$ and $T$ but depends on $\epsilon$ and $A$. 
Note that $T^{-A} n\sum_{m>n} \frac{|C_\phi(n+m)|}{m^{3/2+\epsilon}}$ only appears when $\phi$ is non-holomorphic.

When $t=0$ the main term $M(1/2, \phi; n)$ (when $\xi=1_N$) is given by 
\begin{multline}
M(1/2, \phi; n) 
= \frac{C_\phi(n)}{\sqrt{n}} \frac{\varphi(M)}{M} 
\frac{1}{\pi^2} \int_{-\infty}^\infty h(r) r \tanh(\pi r) 
\bigg(\psi\left(\frac{k}{2}+ir\right) + \psi\left(\frac{k}{2}-ir\right)\bigg)
\; dr
\\ + \frac{C_\phi(n)}{\sqrt{n}} \frac{\varphi(M)}{M} 
\bigg\{\sum_{p\mid M} \frac{\log p}{1-p^{-1}} + \log \left(\frac{M}{(2\pi)^2 n}\right) + 2\gamma_0 \bigg\} 
H_0(h)
\end{multline}
when $\phi$ is a holomorphic cusp form of weight $k$ (even), 
and 
\begin{multline}
M(1/2, \phi; n) 
= \frac{C_\phi(n)}{\sqrt{n}} \frac{\varphi(M)}{M} 
\frac{1}{\pi^2} \int_{-\infty}^\infty h(r) r \tanh(\pi r) 
\\ \times 
\frac{1}{2} \bigg(\psi\left(\frac{1}{2}+\nu+ir\right) + \psi\left(\frac{1}{2}+\nu-ir\right)
+ \psi\left(\frac{1}{2}-\nu+ir\right) + \psi\left(\frac{1}{2}-\nu-ir\right)\bigg)
\; dr
\\ + \frac{C_\phi(n)}{\sqrt{n}} \frac{\varphi(M)}{M} 
\bigg\{\sum_{p\mid M} \frac{\log p}{1-p^{-1}} + \log \left(\frac{M}{(2\pi)^2 n}\right) + 2\gamma_0 \bigg\} 
H_0(h)
\end{multline}
when $\phi$ is a non-holomorphic automorphic cusp form of type $\nu\in i\R$. 
Here $\gamma_0$ is the Euler-Mascheroni constant and $\psi(x) = \frac{\Gamma'}{\Gamma}(x)$ is the digamma function.   

When $t \ne 0$, $M(1/2 +it, \phi; n)$ is given by \eqref{e:M_holo} when $\phi$ is holomorphic and  \eqref{e:Mdef_nonholo} when $\phi$ is non-holomorphic.

For  $T \gg n$, the main term $M(1/2, \phi; n)$ is asymptotic to  
\begin{equation}\label{e:main_asymp}
\frac{c_\alpha C_\phi (n)}{\sqrt{n}}\frac{\varphi(M)}{M} T^{1+\alpha}\bigg(2 \log T +
\sum_{p\mid M} \frac{\log p}{1-p^{-1}} +  \log \left(\frac{M}{(2\pi)^2n}\right)+2\gamma_0 \bigg)
\end{equation}
for an explicit constant $c_\alpha$ dependent only on $\alpha$, as $T \rightarrow \infty$.
When $t \ne 0$, $s = 1/2 +it$, and  $T \gg n$, the main term $M(1/2 +it, \phi; n)$, given by \eqref{e:M_holo} when $\phi$ is holomorphic, and   \eqref{e:Mdef_nonholo} when $\phi$ is non-holomorphic, is asymptotic to 
\begin{equation}\label{e:main_asymp2}
d_\alpha L^{(M)}(1+2it, \xi) \frac{C_\phi(n)}{n^{\frac{1}{2}+it} }T^{1+\alpha},
\end{equation}
for another explicit non-zero constant $d_\alpha$ dependent only on $\alpha$, as $T \rightarrow \infty$.
(Note that $\phi$, and hence $M, \kappa$, are fixed as $T \rightarrow \infty$.)

The main term disappears if $C_{\phi}(n)=0$.   If $C_{\phi}(n) \ne 0$ then the main term dominates the error term (i.e. the first moment expression given in \eqref{e:first_est} is non-trivial) for all $\beta <1$ and $0<\alpha \le 1$ when
\begin{equation}\label{bound}
C_{\phi}(n)T \gg \left( M^\epsilon 2^{\frac{k}{2}} n^{1+ \theta+\epsilon} \right)^{\frac{1}{\min (\alpha, 1-\beta)}+ \epsilon},
\end{equation}
with $0 < \epsilon < \min (\alpha, 1-\beta)$.
Recall $\theta =0$ when $\phi$ is a holomorphic cusp form, and $k = 0$, $\theta$ is the best progress toward the Ramanujan conjecture, when $\phi$ is a weight 0 Mass cusp form.
\end{theorem}

\begin{remark}
The test function $h_{T, \alpha}(r)$ isolates Maass cusp forms $u_j$ with the Laplace eigenvalue $\frac{1}{4}+r_j^2$ in the range $\left|T-|r_j|\right| \ll T^\alpha$, and by Weyl's law  the number of such Maass cusp forms  is a multiple of $T^{1+\alpha}$.  Thus the theorem implies the Lindel\"of Hypothesis is true on average for   $\scrL(1/2+it, \phi\times \overline{u_j}) $ in the spectral aspect, for $|T-|r_j|| \ll T^{\alpha + \epsilon}$ when $T$ satisfies the lower bound \eqref{bound}.
\end{remark}

A consequence of Theorem~\ref{thm:upperbound_first} is the following corollary. 
\begin{corollary}\label{cor:determine}
For $\ell\in \{1, 2\}$, let $\phi_\ell$ be a newform of even weight $k_\ell$ if holomorphic and type $\nu_\ell$ if non-holomorphic, and level $N_\ell$. 
We further assume that $\phi_\ell$ is normalized so the first Fourier coefficient is $1$. 
Let $k=\max\{k_1, k_2\}$, $\kappa = 1$ if  $\phi_\ell$ are holomorphic, and $k = 0$,  $\kappa=\max\{|\nu_1|, |\nu_2|\}$ if  the $\phi_\ell$ are non-holomorphic.
Finally, let $M=\lcm(N_1, N_2)$ and let $T\gg_\epsilon (2^{\frac{k}{2}}M \kappa^2)^{1+\epsilon}$. 
Let $\{u_j\}_{j\geq 1}$ be an orthonormal basis of Hecke-Maass cusp forms for $\Gamma_0(M)$. 
If 
\begin{equation}\label{above}
\scrL(1/2, \phi_1 \times u_j) = \scrL(1/2, \phi_2 \times u_j), 
\end{equation}
for $|r_j| \ll T^{1+\epsilon}$, 
then $\phi_1=\phi_2$.
\end{corollary}
The proof is given in \S\ref{ss:proof_cor_detrmine}.
This is the latest incarnation of a theorem that was proved first by Luo and Ramakrishnan \cite{LR97}. 
They showed that if
\begin{equation}
L(1/2, \phi_1 \times \chi_d) = L(1/2, \phi_2 \times \chi_d) 
\end{equation}
for all quadratic characters $\chi_d$, then $\phi_1=\phi_2$.  
Generalizations of this appeared in 
\cite{Luo99},  \cite{CD05},  \cite{GHS09}  and \cite{Zha11}.
The most recent results that we are aware of are \cite{ms15} and \cite{ss19}. 
In \cite{ms15}, they prove a result similar to Corollary~\ref{cor:determine}. 
For Maass newforms of full level $\phi_1$ and $\phi_2$ 
they show that if \eqref{above} holds for 
$|r_j| \ll_\epsilon \kappa^{4\theta +4 + \epsilon}$ 
then $\phi_1=\phi_2$, where $\theta$ refers to the best progress toward the Ramanujan Conjecture for Mass forms of full level.  Here we work with arbitrary level, and both holomorphic and Maass forms, have improved the exponent of $\kappa$ from $4$ to $2$ 
and eliminated the dependence on the Ramanujan Conjecture.
Note that in the Maass form case the term $2^{k/2}$ is replaced by 1 in the lower bound for $T$.
\section{A first moment formula}
We use the same notation as in \S\ref{ss:notation} and \S\ref{ss:spectral_decomp}. 
Recalling \eqref{e:K_firstmoment_def}, we consider 
\begin{multline}
K(s, \phi; n, h)
= \sum_j \frac{h(r_j)}{\cosh(\pi r_j)} \rho_j(n) \scrL(s, \phi\times \overline{u_j}) 
\\ + \sum_{\cuspa} \frac{1}{4\pi} \int_{-\infty}^\infty \frac{h(r)}{\cosh(\pi r)} 
\tau_\cuspa(1/2+ir, n) \scrL_\cuspa(s, ir; \phi) \; dr.
\end{multline}
For $\Re(s)>1$, by opening up the Rankin-Selberg $L$-functions we get
\begin{multline}\label{e:K_firstmoment_open}
K(s, \phi; n, h)
=
L^{(M)}(2s, \xi) 
\sum_{m=1}^\infty \frac{C_\phi(m)}{m^s} 
\bigg\{\sum_j \frac{h(r_j)}{\cosh(\pi r_j)} \overline{\rho_j(m)} \rho_j(n) 
\\ + \sum_{\cuspa} \frac{1}{4\pi} \int_{-\infty}^\infty \frac{h(r)}{\cosh(\pi r)} 
\overline{\tau_{\cuspa}(1/2+ir, m)} \tau_\cuspa(1/2+ir, n) \;dr\bigg\}. 
\end{multline}
For each $m\geq 1$, we now apply the Bruggeman-Kuznetsov trace formula.

\subsection{The Bruggeman-Kuznetsov trace formula}\label{ss:BK_trace}
For a non-zero integer $q$ and $m, n\in \Z$, let 
\begin{equation}\label{e:Kloosterman}
S(m, n; q) = \sum_{\substack{a\bmod{q}, \\ a\bar{a}\equiv 1\bmod{q}}} e^{2\pi i \frac{ma+n\bar{a}}{q}}
\end{equation}
denote the Kloosterman sum. 

From \cite[Theorem~9.2]{Iwa02}, we have the following. 
For non-zero integers $m$ and $n$, assume $mn>0$.  
For any $h(t)$ satisfying the conditions in \S\ref{ss:first}, we have
\begin{multline}\label{e:Kuznetsov}
\sum_j \frac{h(r_j)}{\cosh(\pi r_j)} \overline{\rho_j(m)}\rho_j(n)
+ \sum_\cuspa \frac{1}{4\pi} \int_{-\infty}^\infty 
\frac{h(r)}{\cosh(\pi r)} \overline{\tau_\cuspa\left(1/2+ir, m\right)}\tau_\cuspa\left(1/2+ir, n \right) \; dr
\\ = \delta_{m=n} \cdot H_0(h)
+ \sum_{q\equiv 0 \bmod M}
\frac{S(m, n; q)}{q} H^{+} \left(\frac{4\pi\sqrt{mn}}{q}; h\right),
\end{multline}
where $H_0$ is given in \eqref{e:H0} 
and 
\begin{equation}
H^+(x; h) = \frac{2i}{\pi}\int_{-\infty}^\infty J_{2it}(x)\frac{h(t)t}{\cosh(\pi t)}\; dt.
\end{equation}
Here $J_{2it}(x)$ is the classical $J$-Bessel function. 

By applying the Bruggeman-Kuznetsov trace formula \eqref{e:Kuznetsov} to \eqref{e:K_firstmoment_open}, 
for $\Re(s)>1$, we get
\begin{equation}\label{e:K_open}
K(s, \phi; n, h)
= L^{(M)}(2s, \xi) \frac{C_\phi(n)}{n^s} H_0(h)
+ K_0(s, \phi; n, h)
\end{equation}
where 
\begin{equation}\label{e:K0_def}
K_0(s, \phi; n, h)
= L^{(M)}(2s, \xi) \sum_{q\equiv 0\bmod{M}} 
\sum_{m=1}^\infty \frac{C_\phi(m)}{m^s} \frac{S(m, n; q)}{q} H^+\left(\frac{4\pi \sqrt{mn}}{q}; h\right).
\end{equation}
The next step is to write $K_0(s, \phi; n, h)$ as an inverse Mellin transform of the series associated to $\phi$ twisted by Kloostermans $S(m, n; q)$ and $J$-Bessel functions. 

For $n\geq 1$ and $q\geq 1$, for $u, s\in \C$, 
we define 
\begin{equation}\label{e:Lq_JBessel}
L_q(s, u; \phi; n) = \sum_{m=1}^\infty \frac{C_\phi(m)}{m^s} \frac{S(m, n; q)}{q} J_{2u}\left(\frac{4\pi \sqrt{mn}}{q}\right), 
\end{equation}
where the $C_\phi(m)$'s are the Fourier coefficients of $\phi$ normalized as in \eqref{e:Fourierexp_holo} and  \eqref{e:Fourierexp_nonholo} (depending on whether $\phi$ is holomorphic or not). 

\begin{lemma}\label{lem:Lq_conv}
The series $L_q(s, u; \phi;n)$ converges absolutely for $\Re(u)\geq 0$ and $\Re(s)>3/4$. 
\end{lemma}

\begin{proof}
By the following crude estimate 
\begin{equation}
J_{2u}(x) \ll \min\{x^{2\Re(u)}, x^{-\frac{1}{2}}\}
\end{equation} 
and Weil's bound for Kloosterman sums 
\begin{equation}\label{e:Weil_Kloosterman}
|S(m, n; q)| \leq \sqrt{q} \sqrt{\gcd(m, n, q)} \sigma_0(q), 
\end{equation}
and as it is well known that 
\begin{equation}
\sum_{1\leq m \leq X} \left|C_\phi(m)\right| \ll_\phi X, 
\end{equation}
for any sufficiently large $X$, we get
\begin{equation}
L_q(s, u; \phi; n) \ll_{q, n,\phi} 
\sum_{m=1}^\infty \frac{1}{m^{\Re(s)+\frac{1}{4}}} < \infty 
\end{equation}
for $\Re(s) >3/4$. 
\end{proof}

We will now write $K_0(s, \phi; n, h)$ as an inverse Mellin transform of $L_q(s, u; \phi, n)$. 
First for sufficiently small  $\sigma_u>0$, we get
\begin{equation}\label{e:H+_mellin}
H^+(x; h) = \frac{2i}{\pi} \int_{-\infty}^\infty J_{2it}(x) \frac{h(t) t}{\cosh(\pi t)} \; dt
= \frac{4}{2\pi i} \int_{(\sigma_u)} J_{2u}(x) \frac{h(u/i) u }{\cos(\pi u)} \; du. 
\end{equation}
Since the integral and the series $L_q(s, u; \phi; n)$ converge absolutely for $\Re(s)> 3/4$, 
we change the order and get
\begin{equation}\label{e:K0_open} 
K_0(s, \phi; n, h)
= L^{(M)}(2s, \xi) 
\sum_{q\equiv 0\bmod{M}} 
\frac{4}{2\pi i} \int_{(\sigma_u)}\frac{h(u/i) u}{\cos(\pi u)} L_q(s, u; \phi; n) \; du. 
\end{equation}

The next section will be devoted  to deriving the analytic properties of $L_q(s, u; \phi; n)$.  
These will be summarized in Proposition~\ref{prop:Lq_nonholo} and then applied to \eqref{e:K0_open}.  
Note also that, as we will show in \eqref{e:Lq_open1_qbound}, 
\begin{equation}
L_q(s, u; \phi, n) \ll q^{-\frac{1}{2}+2\sigma_v}
\end{equation}
for $\Re(u) >0$, $-\Re(u) < \sigma_v < 0$ and $\Re(s)+\sigma_v>1$. 
For this reason, we choose $\Re(u)=\sigma_u >1/4$ and $-\sigma_u< \sigma_v <-1/4$ so the series over $q$ converges absolutely 
and we may interchange the summation and the integral, obtaining 
\begin{equation}\label{e:K0_after_interchange}
K_0(s, \phi; n, h)
= L^{(M)}(2s, \xi) 
\frac{4}{2\pi i} \int_{(\sigma_u)}\frac{h(u/i) u}{\cos(\pi u)} \sum_{q\equiv 0\bmod{M}} L_q(s, u; \phi; n) \; du. 
\end{equation}
Here we further assume that $\frac{1}{4} < \sigma_u < \frac{3}{2}$; otherwise we pass over the pole of $\frac{1}{\cos(\pi u)}$ at $u=\frac{3}{2}$. 

\section{Additive twists and Kloosterman sums twists of $L$-series}\label{s:additive}

Let $\phi$ be an automorphic form of weight $k$, type $\nu$ and level $N$, with central character $\xi$, as given in \S\ref{s:intro}. 
For $n, q\geq 1$ with $N\mid q$, for $u, s\in \C$ with $\Re(u)\geq 0$ and $\Re(s)> 3/4$, 
recalling \eqref{e:Lq_JBessel}, the $L$-series twisted by Kloosterman sums and $J$-Bessel functions is given by 
\begin{equation}
L_q(s, u; \phi; n) = \sum_{m=1}^\infty \frac{C_\phi(m)}{m^s} \frac{S(m, n; q)}{q} J_{2u}\left(\frac{4\pi \sqrt{mn}}{q}\right)
\end{equation}
and this series converges absolutely in this region. 

 For $\alpha\in \Q$, we define the additive twist of the $L$-function as follows: 
\begin{equation}\label{e:additive_L}
L(s, \phi; \alpha) = \sum_{n=1}^\infty \frac{C_\phi(n) e(n\alpha)}{n^s}.
\end{equation}
Our aim in this section 
is to obtain the analytic properties of $L_q(s, u; \phi; n)$ 
(see Proposition~\ref{prop:Lq_nonholo}). 
By using the Mellin inversion formula for the $J$-Bessel function \eqref{e:JBessel_MI} 
and also the definition of Kloosterman sums, 
we write the $L$-series $L_q(s, u; \phi; n)$ as a contour integral of a linear combination 
of additive twists of $L$-functions $L(s, \phi; a/q)$ for $a\bmod{q}$, $\gcd(a, q)=1$ (see \eqref{e:Lq_open2}). 
We then apply the functional equation and meromorphic continuation of the additive twists of $L$-functions.

\subsection{$L$-series twisted by Kloosterman sums and $J$-Bessel functions}\label{ss:L-K-B}

By the inverse Mellin transform \cite[6.422(6)]{GR}, 
for a fixed $u\in \C$ with $\Re(u)>0$, for $x>0$, we have
\begin{equation}\label{e:JBessel_MI}
J_{2u}(x) = \frac{1}{2\pi i}\int_{(\sigma_v)} \frac{\Gamma\left(u+v\right)}{\Gamma\left(1+u-v\right)} \left(\frac{x}{2}\right)^{-2v} \; dv
\end{equation}
for $-\Re(u) < \sigma_v < \frac{1}{2}$. 
Note that when $-\Re(u)< \sigma_v< 0$, by Stirling's formula, the integral converges absolutely. 

For $\Re(u)>0$, $-\Re(u) < \sigma_v< 0$ and $\Re(s)+\sigma_v>1$, we apply \eqref{e:JBessel_MI} to \eqref{e:Lq_JBessel} and get
\begin{equation}\label{e:Lq_open1}
L_q(s, u; \phi; n)
= \frac{1}{2\pi i} \int_{(\sigma_v)}  \frac{\Gamma\left(u+v\right)}{\Gamma\left(1+u-v\right)} 
(2\pi)^{-2v} n^{-v} \sum_{m=1}^\infty \frac{C_\phi(m)}{m^{s+v}} \frac{S(m, n; q)}{q^{1-2v}} 
\; dv
\end{equation}
By the Weil bound for Kloosterman sums \eqref{e:Weil_Kloosterman}, we see that 
\begin{equation}\label{e:Lq_open1_qbound}
L_q(s, u; \phi; n) \ll_n q^{-\frac{1}{2}+2\sigma_v}.
\end{equation}
After opening up the Kloosterman sums \eqref{e:Kloosterman}, we get
\begin{equation}\label{e:Lq_open2}
L_q(s, u; \phi; n)
= \frac{1}{2\pi i} \int_{(\sigma_v)}  \frac{\Gamma\left(u+v\right)}{\Gamma\left(1+u-v\right)} 
(2\pi)^{-2v} n^{-v} q^{-1+2v} 
\sum_{\substack{a\bmod{q}, \\ a\bar{a}\equiv 1\bmod{q}}} e^{2\pi in \frac{a}{q}} L(s+v, \phi; \bar{a}/q)\; dv.
\end{equation}

We will now obtain the following analytic properties of $L(s, \phi; a/q)$. 
When the level of $\phi$ (in this case $N$) divides $q$ the proof is quite well-known (e.g. see \cite[Theorem~3.1]{DHKL}) 
and we skip the proof. 

\begin{proposition}\label{prop:L_fe}
Let $q$ and $N$ be positive integers and assume that $q$ is divisible by $N$. 
Using the notation in \S\ref{ss:notation}, let 
$\phi$ be an automorphic form of level $N$, which is either a holomorphic cusp form of weight $k$, $\nu=\pm \frac{k-1}{2}$, 
or a non-holomorphic Maass form of (weight $0$) type $\nu$, 
with Fourier expansions given in \eqref{e:Fourierexp_holo} and \eqref{e:Fourierexp_nonholo} respectively.
For $a\in \Z$ with $\gcd(a, q)=1$, the twisted $L$-series 
\begin{equation}
L(s, \phi; a/q) = \sum_{n=1}^\infty \frac{C_\phi(n)e\left(n\frac{a}{q}\right)}{n^s}
\end{equation}
has a meromorphic continuation to $s\in \C$ 
with possible simple poles only at $s=1\pm \nu$ 
with residues 
\begin{equation}\label{e:L_res}
\Res_{s=1\pm \nu} L(s, \phi; a/q)
= 2\xi(d) c_\phi^\pm\frac{q^{-1\mp 2\nu} \pi^{\frac{1}{2}\pm \nu}}
{\Gamma\left(\frac{1}{2}\pm \nu\right)}.
\end{equation}
Moreover it satisfies the following functional equation:
when $\phi$ is a holomorphic cusp form of weight $k$: 

\begin{equation}\label{e:L_fe_holo}
L(s, \phi;a/q) 
=  q^{-2s+1} (2\pi)^{2s-1} i^k \xi(d)
\frac{\Gamma\left(1-s+\frac{k-1}{2}\right)}{\Gamma\left(s+\frac{k-1}{2}\right)} 
L(1-s, \phi; -d/q).
\end{equation}
When $\phi$ is a non-holomorphic automorphic form, 
\begin{multline}\label{e:L_fe_nonholo}
L(s, \phi; a/q)
= \xi(d) q^{-2s+1}(2\pi)^{2s-1} 
\Gamma\left(1-s-\nu\right)\Gamma\left(1-s+\nu\right)
\\ \times \bigg\{ -\frac{1}{\Gamma\left(\frac{1}{2}+s\right)\Gamma\left(\frac{1}{2}-s\right)} L(1-s, \phi; -d/q) 
+ c_\phi(-1)\frac{\cos(\pi \nu)}{\pi} L(1-s, \phi; d/q) \bigg\}.
\end{multline}
Here $ad\equiv 1\bmod{q}$. 
\end{proposition}


We apply the meromorphic continuation and functional equation to $L(s+v, \phi; a/q)$ in $L_q(s, u; \phi; n)$, 
and decompose $L_q(s, u; \phi; n)$ into several pieces. 

We define the generalized Gauss sum: 
\begin{equation}
c_{\xi|q}(m) = \sum_{\substack{a\bmod{q}, \\ \gcd(a, q)=1}} 
\xi(a) e^{2\pi i m\frac{a}{q}}. 
\end{equation}
When $\xi$ is a trivial character, write $c_{q}(m) = c_{\xi|q}(m)$ which is the usual Ramanujan sum.

Fix $x>0$. 
We define the following two integrals for $s, u\in \C$, $1/2 < \Re(s) < \Re(u)$.
First, for $\nu\in i\R$:
\begin{equation}\label{e:F1_def}
F_1(s, u, \nu; x) = \frac{1}{2\pi i} \int_{(-\Re(s)-\epsilon)}  
\frac{\Gamma\left(u+v\right)}{\Gamma\left(1+u-v\right)} 
\frac{\Gamma\left(1-s-v-\nu\right)\Gamma\left(1-s-v+\nu\right)}
{\Gamma\left(\frac{1}{2}+s+v\right)\Gamma\left(\frac{1}{2}-s-v\right)}x^v \; dv
\end{equation}
and 
\begin{equation}\label{e:F3_def}
F_3(s, u, \nu; x)
= \frac{1}{2\pi i} \int_{(\Re(s)+\epsilon)} 
\frac{\Gamma\left(u-v\right) \Gamma\left(1-s-\nu+v\right) \Gamma\left(1-s+\nu+v\right)}
{\Gamma\left(1+u+v\right)} x^v\; dv.
\end{equation}
Then, for a positive integer $k$, we define 
\begin{equation}\label{e:F2_def}
F_2(s, u, k; x) = -i^k \frac{1}{2\pi i} \int_{(-\Re(s)-\epsilon)}  
\frac{\Gamma\left(u+v\right)}{\Gamma\left(1+u-v\right)} 
\frac{\Gamma\left(1-s-v+\frac{k-1}{2}\right)}{\Gamma\left(s+v+\frac{k-1}{2}\right)} x^v \; dv. 
\end{equation}
Here $\epsilon>0$ is chosen such that $-\Re(u) < \Re(v)=-\Re(s)-\epsilon < 1-\Re(s)$, 
so the contour line $\Re(v)=-\Re(s)-\epsilon$ separates the poles of gamma functions inside the integrals. 
Note that, by Stirling's bound for gamma functions, the integrals defining $F_1(s, u, \nu; x)$ and $F_2(s, u, k; x)$ 
converge absolutely for $\Re(s)>1/2$.  
Moreover, when $k$ is even, $F_1\left(s, u, \frac{k-1}{2}; x\right) = F_2\left(s, u, k; x\right)$ as expected.

 We prove that $F_1$, $F_3$ and $F_2$ are, up to some ratio of gamma functions, hypergeometric functions 
in Lemmas~\ref{lem:F1} and \ref{lem:F2}, 
which implies that $F_1$, $F_3$ and $F_2$ have meromorphic continuations to $s, u\in \C$. 
Their integral representations are given in Corollaries~\ref{cor:F1} and \ref{cor:F2}.

In the following proposition we describe $L_q(s, u; \phi; n)$ as a series  using the Fourier coefficients of $\phi$ 
and $F_1(s, u, \nu; x)$ and $F_3(s, u, \nu; x)$ when $\phi$ is non-holomorphic, 
and $F_2(s, u, k; x)$ when $\phi$ is holomorphic. 

\begin{proposition}\label{prop:Lq_nonholo}
Let $\phi$ be either a non-holomorphic form of type $\nu\in i\R$ or a holomorphic cusp form of weight $k$. 
When $\phi$ is a holomorphic cusp form, we set $\nu=\frac{k-1}{2}$. 

For $1/2\leq \Re(s) < \Re(u)$, we have
\begin{multline}\label{e:Lq_prop}
L_q(s, u; \phi; n)
= M_q^+(s, u; \phi; n) + M_q^-(s, u; \phi; n) + \delta_{\xi=1_N} M_q(s, u; \phi; n)
\\ + L_{q, 1}^+(s, u; \phi; n)+L_{q, 1}^-(s, u; \phi; n)+ L_{q, 2}(s, u; \phi; n). 
\end{multline}
Here 
\begin{equation}\label{e:Mq+}
M_q^+(s, u; \phi; n)
= 2 c_\phi^+ 
\frac{\Gamma\left(u-s+1+\nu\right)}{\Gamma\left(u+s-\nu\right)} 
\frac{\pi^{\frac{1}{2}+\nu}(2\pi)^{2s-2-2\nu} }{\Gamma\left(\frac{1}{2}+\nu\right)}
n^{s-1-\nu} 
\frac{c_{\xi|q}(n)}{q^{2s}}
\end{equation}
and 
\begin{equation}\label{e:Mq-}
M_q^-(s, u; \phi; n)
= 2c_\phi^-
\frac{\Gamma\left(u-s+1-\nu\right)}{\Gamma\left(u+s+\nu\right)} 
\frac{\pi^{\frac{1}{2}-\nu}(2\pi)^{2s-2+2\nu}}{\Gamma\left(\frac{1}{2}-\nu\right)} 
n^{s-1+\nu} 
\frac{c_{\xi|q}(n)}{q^{2s}}. 
\end{equation}
When $\phi$ is non-holomorphic, 
\begin{multline}\label{e:Mq_nonholo}
M_q(s, u; \phi; n)
\\ = -(2\pi)^{2s-1} \frac{\varphi(q)}{q^{2s}} \frac{C_\phi(n)}{n^{1-s}} 
\frac{\Gamma\left(1-s-\nu+u\right) \Gamma\left(1-s+\nu+u\right)}
{\Gamma\left(\frac{1}{2}-s+u\right) \Gamma\left(\frac{1}{2}+s-u\right)}
\frac{\Gamma\left(2s-1\right)}
{\Gamma\left(s+\nu+u\right)\Gamma\left(s-\nu+u\right)}. 
\end{multline}
When $\phi$ is holomorphic, 
\begin{multline}\label{e:Mq_holo}
M_q(s, u; \phi; n)
\\ = (2\pi)^{2s-1} \frac{\varphi(q)}{q^{2s}} \frac{C_\phi(n)}{n^{1-s}} 
i^k \Gamma\left(2s-1\right) \frac{\sin\left(\pi\left(s+u-\frac{k-1}{2}\right)\right)}{\pi}
\frac{\Gamma\left(1-s-u+\frac{k-1}{2}\right)\Gamma\left(1-s+u+\frac{k-1}{2}\right) }
{\Gamma\left(s+u+\frac{k-1}{2}\right)\Gamma\left(s-u+\frac{k-1}{2}\right)}.
\end{multline}

When $\phi$ is non-holomorphic, 
\begin{equation}\label{e:Lq1+nonholo}
L_{q, 1}^+(s, u; \phi; n)
= - (2\pi)^{2s-1} 
q^{-2s} \xi(-1) \sum_{m=1}^\infty \frac{C_\phi(m+n)c_{\xi|q}(m)}{(m+n)^{1-s}} 
F_1\left(s, u, \nu; \frac{m+n}{n}\right). 
\end{equation}
When $\phi$ is holomorphic, 
\begin{equation}\label{e:Lq1+holo}
L_{q, 1}^+(s, u; \phi; n)
= -(2\pi)^{2s-1} q^{-2s} \xi(-1)\sum_{m=1}^\infty \frac{C_\phi(m+n) c_{\xi|q}(m)}{(m+n)^{1-s}} F_2\left(s, u, k; \frac{m+n}{n}\right). 
\end{equation}
When $\phi$ is non-holomorphic, 
\begin{equation}\label{e:Lq1-nonholo}
L_{q, 1}^{-}(s, u; \phi; n)
= -(2\pi)^{2s-1} q^{-2s} \sum_{m=1}^{n-1} \frac{C_\phi(m) c_{\xi|q}(n-m)}{m^{1-s}} F_1\left(s, u, \nu; \frac{m}{n}\right).
\end{equation}
When $\phi$ is holomorphic, 
\begin{equation}\label{e:Lq1-holo}
L_{q, 1}^-(s, u; \phi; n)
= -(2\pi)^{2s-1} q^{-2s}\sum_{m=1}^{n-1} \frac{C_\phi(m)c_{\xi|q}(n-m)}{m^{1-s}} F_2\left(s, u, k; \frac{m}{n}\right). 
\end{equation}
Finally, when $\phi$ is non-holomorphic,  
\begin{equation}\label{e:Lq2}
L_{q, 2}(s, u; \phi; n)
= (2\pi)^{2s-1} c_\phi(-1)\frac{\cos(\pi \nu)}{\pi} q^{-2s}
\sum_{m=1}^\infty \frac{C_\phi(m)c_{\xi|q}(n+m)}{m^{1-s}} F_3\left(s, u, \nu; \frac{n}{m}\right), 
\end{equation}
for some $\epsilon>0$.
Note that, by \eqref{e:F1_x>1_bound}, \eqref{e:F3_bound} and \eqref{e:F2_x>1_bound}, 
the series converge absolutely for $\Re(s)< \Re(u)$.

When $\phi$ is holomorphic, $M_q^+ = M_q^-=L_{q, 2}=0$. 
\end{proposition}

\subsection{Proof of 
Proposition~\ref{prop:Lq_nonholo}}
We start with \eqref{e:Lq_open2}.
For the readers' convenience, we recall the formula: 
\begin{equation}
L_q(s, u; \phi; n)
= \frac{1}{2\pi i} \int_{(\sigma_v)}  \frac{\Gamma\left(u+v\right)}{\Gamma\left(1+u-v\right)} 
(2\pi)^{-2v} n^{-v} 
\sum_{\substack{a\bmod{q}, \gcd(a, q)=1, \\ a\bar{a}\equiv 1\bmod{q}}} 
\frac{e^{2\pi i n\frac{a}{q}}}{q^{1-2v}} L(s+v, \phi; \bar{a}/q)
\; dv.
\end{equation}
 We first check that the region containing $s, u\in \C$, where both of the integral and inner series converges absolutely, is non-empty. 
The initial assumptions are $\Re(s)>1$, $\frac{1}{4}< \Re(u) < \frac{3}{2}$, $-\Re(u) < \sigma_v < 0$ and $\Re(s)+\sigma_v>1$. 
Under these assumptions the series $L(s+v, \phi; \bar{a}/q)$ and the $v$-integral converge absolutely. 
We further assume that $\Re(u)> \Re(s)$ for the later purpose. 
To make the region for $s\in \C$ non-empty, 
we choose $\sigma_v = -\epsilon$, $1+2\epsilon< \Re(u) < \frac{3}{2}$.
Then $1-\sigma_v = 1+\epsilon < \Re(s) < \Re(u)< \frac{3}{2}$. 

We now move the $v$ line of integration to $-\Re(u) < \Re(v)< -\Re(s)$. 
By Proposition~\ref{prop:L_fe}, the additive twist $L(s, \phi; \bar{a}/q)$ has meromorphic continuation to $s\in \C$ 
with possible simple poles only at $s=1\pm \nu$. 
So we get 
\begin{multline}
L_q(s, u; \phi; n)
\\ = 
\frac{\Gamma\left(u-s+1+\nu\right)}{\Gamma\left(u+s-\nu\right)} 
(2\pi)^{2s-2-2\nu} n^{s-1-\nu} 
\sum_{\substack{a\bmod{q}, \gcd(a, q)=1, \\ a\bar{a}\equiv 1\bmod{q}}} 
\frac{e^{2\pi i n\frac{a}{q}}}{q^{-1+2s-2\nu}} \Res_{v=-s+1+\nu} L(s+v, \phi; \bar{a}/q)
\\ +
\frac{\Gamma\left(u-s+1-\nu\right)}{\Gamma\left(u+s+\nu\right)} 
(2\pi)^{2s-2+2\nu} n^{s-1+\nu} 
\sum_{\substack{a\bmod{q}, \gcd(a, q)=1, \\ a\bar{a}\equiv 1\bmod{q}}} 
\frac{e^{2\pi i n\frac{a}{q}}}{q^{-1+2s+2\nu}} \Res_{v=-s+1-\nu} L(s+v, \phi; \bar{a}/q)
\\ + \frac{1}{2\pi i} \int_{(-\Re(s)-\epsilon)}  \frac{\Gamma\left(u+v\right)}{\Gamma\left(1+u-v\right)} 
(2\pi)^{-2v} n^{-v} 
\sum_{\substack{a\bmod{q}, \gcd(a, q)=1, \\ a\bar{a}\equiv 1\bmod{q}}} 
\frac{e^{2\pi i n\frac{a}{q}}}{q^{1-2v}} L(s+v, \phi; \bar{a}/q)
\; dv.
\end{multline}
 By applying the residues given in \eqref{e:L_res}, we get \eqref{e:Mq+} and \eqref{e:Mq-}. 
Note that these residues vanish when $\phi$ is holomorphic. 

We define  
\begin{multline}\label{e:Lq0_def}
L_{q, 0}(s, u; \phi; n)
\\ = \frac{1}{2\pi i} \int_{(-\Re(s)-\epsilon)}  \frac{\Gamma\left(u+v\right)}{\Gamma\left(1+u-v\right)} 
(2\pi)^{-2v} n^{-v} 
\sum_{\substack{a\bmod{q}, \gcd(a, q)=1, \\ a\bar{a}\equiv 1\bmod{q}}} 
\frac{e^{2\pi i n\frac{a}{q}}}{q^{1-2v}} L(s+v, \phi; \bar{a}/q) \; dv. 
\end{multline}
We now apply the functional equation of additive twists $L$-functions in Proposition~\ref{prop:L_fe}. 
When $\phi$ is a non-holomorphic automorphic form, 
by applying \eqref{e:L_fe_nonholo} to \eqref{e:Lq0_def}, we get
\begin{multline}
L_{q, 0}(s, u; \phi; n)
= 
- (2\pi)^{2s-1} \frac{1}{2\pi i} \int_{(-\Re(s)-\epsilon)}  
\frac{\Gamma\left(u+v\right)}{\Gamma\left(1+u-v\right)} 
\frac{\Gamma\left(1-s-v-\nu\right)\Gamma\left(1-s-v+\nu\right)}
{\Gamma\left(\frac{1}{2}+s+v\right)\Gamma\left(\frac{1}{2}-s-v\right)}
\\ \times 
q^{-2s} \sum_{\substack{a\bmod{q}, \\ \gcd(a, q)=1}} 
e^{2\pi i n\frac{a}{q}} \xi(a)
n^{-v} L(1-s-v, \phi; -a/q) \; dv
\\ + 
(2\pi)^{2s-1}c_\phi(-1)\frac{\cos(\pi \nu)}{\pi} 
\frac{1}{2\pi i} \int_{(-\Re(s)-\epsilon)}  
\frac{\Gamma\left(u+v\right)}{\Gamma\left(1+u-v\right)} \Gamma\left(1-s-v-\nu\right)\Gamma\left(1-s-v+\nu\right)
\\ \times q^{-2s} \sum_{\substack{a\bmod{q}, \\ \gcd(a, q)=1}} 
e^{2\pi i n\frac{a}{q}} \xi(a)
n^{-v} L(1-s-v, \phi; a/q) 
\; dv.
\end{multline}
As $\Re(s+v)=-\epsilon<0$ the additive $L$-series $L(1-s-v, \phi; \pm a/q)$ converges absolutely. 
Moreover, by Stirling's formula, 
the first integral converges absolutely for $\Re(s)>\frac{1}{2}$ and the second one converges absolutely 
for any $s, u\in \C$. 
By changing the variable $v$ to $-v$ in the second integral, we get $L_{q, 2}(s, u; \phi; n)$ given in \eqref{e:Lq2}. 
Recalling the definition of the function $F_1(s, u, \nu; x)$ given in \eqref{e:F1_def} for the first series, 
by changing the order of integral and the series, we get 
%
\begin{equation}
L_{q, 0}(s, u;\phi;n)
= - (2\pi)^{2s-1} q^{-2s} \sum_{m=1}^\infty \frac{C_\phi(m)c_{\xi|q}(n-m)}{m^{1-s}} F_1\left(s, u, \nu; \frac{m}{n}\right)
+ L_{q, 2}(s, u; \phi; n).
\end{equation}

 Following Lemma~\ref{lem:F1},
we separate the first series into three pieces:  $n+1\leq m$ and $1\leq m \leq n-1$ and $m=n$. 
For the parts $n+1\leq m$, let 
\begin{equation}\label{e:Lq1+_def}
L_{q, 1}^+(s, u; \phi; n)
= - (2\pi)^{2s-1} 
q^{-2s} \sum_{m=n+1}^\infty \frac{C_\phi(m)c_{\xi|q}(n-m)}{m^{1-s}} 
F_1\left(s, u, \nu; \frac{m}{n}\right)
\end{equation}
and for $1\leq m\leq n-1$, let 
\begin{equation}\label{e:Lq1-_def}
L_{q, 1}^-(s, u; \phi; n)
= - (2\pi)^{2s-1} 
q^{-2s} \sum_{m=1}^{n-1} \frac{C_\phi(m)c_{\xi|q}(n-m)}{m^{1-s}} 
F_1\left(s, u, \nu; \frac{m}{n}\right).
\end{equation}
When $m=n$, and
$\xi=1_N$, the trivial character modulo $N$, we have $c_{\xi|q}(0) = c_q(0) = \varphi(q)$. 
 When $\xi$ is not the trivial character, by orthogonality, $c_{\xi|q}(0)=0$.

When $\xi$ is the trivial character modulo $N$, by applying \eqref{e:F1_x=1} below, we get
\begin{equation}
-(2\pi)^{2s-1} q^{-2s} \frac{C_\phi(n)}{n^{1-s}} F_1(s, u, \nu; 1) \varphi(q)
= M_q(s, u;\phi; n), 
\end{equation}
which is given in \eqref{e:Mq_nonholo}. 
 By applying \eqref{e:F1_x=1}, we get the formula in \eqref{e:Mq_nonholo}.
Combining above together, we get
\begin{multline}\label{e:Lq_nonholo_inter}
L_q(s, u; \phi; n)
= M_q(s, u ;\phi; n) + M_q^+(s, u; \phi; n) + M_q^-(s, u; \phi; n) 
\\ + L_{q, 1}^+(s, u; \phi; n)+L_{q, 1}^-(s, u; \phi; n) + L_{q, 2}(s, u; \phi; n). 
\end{multline}

When $\phi$ is a holomorphic cusp form of weight $k$, $c_\phi(-1)=c_\phi^\pm=0$.
By applying \eqref{e:L_fe_holo}, similarly, we get
\begin{multline}
L_q(s, u; \phi; n)
= L_{q, 0}(s, u;\phi; n)
\\ = 
(2\pi)^{2s-1} i^k
\frac{1}{2\pi i} \int_{(-\Re(s)-\epsilon)}  
\frac{\Gamma\left(u+v\right)}{\Gamma\left(1+u-v\right)} 
\frac{\Gamma\left(1-s-v+\frac{k-1}{2}\right)}{\Gamma\left(s+v+\frac{k-1}{2}\right)} 
q^{-2s} 
\sum_{m=1}^\infty \frac{C_\phi(m) c_{\xi|q}(n-m)}{m^{1-s}} \left(\frac{m}{n}\right)^v 
\; dv
\\ = -(2\pi)^{2s-1} 
q^{-2s} 
\sum_{m=1}^\infty \frac{C_\phi(m) c_{\xi|q}(n-m)}{m^{1-s}} F_2\left(s, u, k; \frac{m}{n}\right). 
\end{multline} 
We separate the series into three pieces: $m\geq n+1$, $1\leq m \leq n-1$ and $m=n$. 
Similar to the non-holomorphic case, we apply the orthogonality of Dirichlet characters 
and \eqref{e:F2_x=1} for $m=n$. 
This gives us
\begin{equation}
L_{q, 0}(s, u; \phi; n)
= M_q(s, u; \phi; n)
+ L_{q, 1}^{+}(s, u; \phi; n) + L_{q, 1}^{-}(s, u; \phi; n), 
\end{equation}
where 
$M_q(s, u; \phi; n)$ as given in \eqref{e:Mq_holo}, 
\begin{equation}\label{e:Lq1+_holo_def}
L_{q, 1}^+(s, u; \phi; n)
= - (2\pi)^{2s-1} 
q^{-2s} \sum_{m=n+1}^\infty \frac{C_\phi(m)c_{\xi|q}(n-m)}{m^{1-s}} 
F_2\left(s, u, k; \frac{m}{n}\right)
\end{equation}
and 
\begin{equation}\label{e:Lq1-_holo_def}
L_{q, 1}^-(s, u; \phi; n)
= - (2\pi)^{2s-1} 
q^{-2s} \sum_{m=1}^{n-1} \frac{C_\phi(m)c_{\xi|q}(n-m)}{m^{1-s}} 
F_2\left(s, u, k; \frac{m}{n}\right).
\end{equation}

 Recall that  
$F_1(s, u, \nu; x)$ and $F_2(s, u, k; x)$ appearing in the above series,  
given as contour integrals as in \eqref{e:F1_def} and \eqref{e:F2_def}, 
converge absolutely when $\Re(s)>1/2$. 
Moreover $F_3(s, u, \nu; x)$ given in \eqref{e:F3_def} converges absolutely for any $s\in \C$. 
In Lemma~\ref{lem:F1} and Lemma~\ref{lem:F2} 
we will write $F_1$, $F_3$ and $F_2$ in terms of hypergeometric series and get meromorphic continuation to $s\in \C$. 
By applying their analytic properties further, 
we will see that the series $L_{q, 1}^{+}(s, u; \phi; n)$ and $L_{q, 2}(s, u;\phi; n)$ converge absolutely for $1/2\leq \Re(s) < \Re(u)$ and complete the proof of Proposition~\ref{prop:Lq_nonholo}.

\subsection{Analytic properties of $F_1(s, u, \nu; x)$, $F_3(s, u, \nu; x)$ and $F_2(s, u, k; x)$}\label{ss:F12}

Our goal in the following lemmas is to establish the analytic properties of 
$F_1$, $F_3$ and $F_2$. 
We first state Lemma~\ref{lem:F1} and Lemma~\ref{lem:F2} and then give a proof for Lemma~\ref{lem:F2}. 
The proof of Lemma~\ref{lem:F1} is almost parallel to the proof of Lemma~\ref{lem:F2}. 

We recall the definitions of $F_1(s, u, \nu; x)$ and $F_3(s, u, \nu; x)$ given in \eqref{e:F1_def} and \eqref{e:F3_def}. 
Fix $x>0$. 
For $s, u\in \C$, $\frac{1}{2}< \Re(s) < \Re(u)$ and $\nu\in i\R$,
\begin{equation}
F_1(s, u, \nu; x) = \frac{1}{2\pi i} \int_{(-\Re(s)-\epsilon)}  
\frac{\Gamma\left(u+v\right)}{\Gamma\left(1+u-v\right)} 
\frac{\Gamma\left(1-s-v-\nu\right)\Gamma\left(1-s-v+\nu\right)}
{\Gamma\left(\frac{1}{2}+s+v\right)\Gamma\left(\frac{1}{2}-s-v\right)}x^v \; dv
\end{equation}
and 
\begin{equation}
F_3(s, u, \nu; x)
= \frac{1}{2\pi i} \int_{(\Re(s)+\epsilon)} 
\frac{\Gamma\left(u-v\right) \Gamma\left(1-s-\nu+v\right) \Gamma\left(1-s+\nu+v\right)}
{\Gamma\left(1+u+v\right)} x^v\; dv.
\end{equation}
for $\epsilon>0$. 
Note that the contour line $\Re(v)=-\Re(s)-\epsilon$ separates the poles of gamma functions. 
By Stirling's formula the integral converges absolutely for $\Re(s)>1/2$. 

\begin{lemma}\label{lem:F1}
We have the following explicit descriptions for $F_1(s, u, \nu; x)$ for $x>1$, $x=1$ and $x<1$
and $F_3(s, u, \nu; x)$ for $x>0$. 

For $x>1$, 
\begin{multline}\label{e:F1_x>1_hyp}
F_1(s, u, \nu; x) 
= \frac{\Gamma\left(1-s-\nu+u\right) \Gamma\left(1-s+\nu+u\right)}
{\Gamma\left(\frac{1}{2}-s+u\right)\Gamma\left(\frac{1}{2}+s-u\right) \Gamma\left(1+2u\right)}
x^{-u} \;_2F_1\left({1-s-\nu+u, 1-s+\nu+u\atop 1+2u}; x^{-1}\right)
\end{multline}
When $x=1$, 
\begin{equation}\label{e:F1_x=1}
F_1(s, u, \nu; 1) 
= \frac{\Gamma\left(2s-1\right)}{\Gamma\left(\frac{1}{2}+s-u\right)\Gamma\left(\frac{1}{2}-s+u\right)} 
\frac{\Gamma\left(1-s-\nu+u\right)\Gamma\left(1-s+\nu+u\right)}
{\Gamma\left(s+\nu+u\right) \Gamma\left(s-\nu+u\right)}.
\end{equation}
For $0<x<1$, 
\begin{multline}\label{e:F1_x<1_hyp}
F_1(s, u, \nu; x) 
=
-\frac{\Gamma\left(\nu\right)}{2^{1-2\nu}\sqrt{\pi}  \Gamma\left(\frac{1}{2}-\nu\right)}
\frac{\Gamma\left(1-s-\nu+u\right)}{\Gamma\left(s+\nu+u\right)}
x^{1-s-\nu}
\;_2F_1\left({1-s-\nu+u, 1-s-\nu-u\atop 1-2\nu}; x\right) 
\\ -\frac{\Gamma\left(-\nu\right)}{2^{1+2\nu} \sqrt{\pi}\Gamma\left(\frac{1}{2}+\nu\right)}
\frac{\Gamma\left(1-s+\nu+u\right)}{\Gamma\left(s-\nu+u\right)}
x^{1-s+\nu} 
\;_2F_1\left({1-s+\nu+u, 1-s+\nu-u\atop 1+2\nu}; x\right).
\end{multline}

For $x>0$, we have 
\begin{equation}\label{e:F3_hyp}
F_3(s, u, \nu; x)
= \frac{\Gamma\left(1-s-\nu+u\right)\Gamma\left(1-s+\nu+u\right)}{\Gamma\left(1+2u\right)}
x^u \;_2F_1\left({1-s-\nu+u, 1-s+\nu+u\atop 1+2u}; -x\right). 
\end{equation}

Since the hypergeometric function $\;_2F_1({a, b\atop c}; x)$ with any fixed branch 
(provided that $x=0, 1$ and $\infty$ are excluded) is a meromorphic function of $a, b$ and $c$ except for possible poles at 
$c\in \Z_{\leq 0}$, 
$F_1(s, u, \nu; x)$ and $F_3(s, u, \nu; x)$ also have meromorphic continuation to $s, u\in \C$. 
\end{lemma}

By applying the integral representations of the hypergeometric function \cite[\S5.6]{DLMF}, 
we prove the following corollary. 
\begin{corollary}\label{cor:F1}
Assume that $\Re(s)\geq 1/2$ and $\Re(s) < \Re(u)$. 
We have the following integral representations of $F_1$ and $F_3$. 

For $x>1$, 
\begin{multline}\label{e:F1_x>1}
F_1(s, u, \nu; x) 
= \frac{\cos(\pi(s-u))}{\pi} 
\frac{\Gamma\left(1-s- \nu+u\right)}
{\Gamma\left(s+ \nu+u\right)}
\\ \times \frac{1}{2\pi i} \int_{(\sigma_v)} \frac{\Gamma\left(1-s+ \nu+v\right) \Gamma\left(s+ \nu+v\right) \Gamma\left(u-v\right)}
{\Gamma\left(1+u+v\right)} x^{1-s+\nu} (x-1)^{-1+s-\nu-v} \; dv, 
\end{multline}
where $-1+\Re(s) < \sigma_v < \Re(u)$. 
%
 We have 
\begin{equation}\label{e:F1_x>1_bound}
\left(\frac{x}{x-1}\right)^{-1+s-\nu} F_1(s, u, \nu; x) \ll (x-1)^{-\Re(u)} 
\end{equation}
For $0<x<1$, 
\begin{multline}\label{e:F1_x<1}
F_1(s, u, \nu; x) 
= -\frac{\sin(\pi(s+\nu+u))}{2\pi \sin(\pi \nu)}
\frac{1}{\Gamma\left(s-\nu-u\right) \Gamma\left(s-\nu+u\right)}
\\ \times \left(\frac{x}{1-x}\right)^{1-s-\nu}
\frac{1}{2\pi i} \int_{C} \Gamma\left(u-v\right) \Gamma\left(-u-v\right) \Gamma\left(s-\nu+v\right) 
\Gamma\left(1-s-\nu+v\right) (1-x)^{-v} \; dv
\\ -\frac{\sin(\pi(s-\nu+u))}{2\pi \sin(-\pi \nu)}
\frac{1}{\Gamma\left(s+\nu-u\right) \Gamma\left(s+\nu+u\right)}
\\ \times \left(\frac{x}{1-x}\right)^{1-s+\nu} 
\frac{1}{2\pi i} \int_{C} \Gamma\left(u-v\right) \Gamma\left(-u-v\right) \Gamma\left(s+\nu+v\right) 
\Gamma\left(1-s+\nu+v\right) (1-x)^{-v} \; dv, 
\end{multline}
where the contour $C$ separates the poles of gamma functions.

For $x>0$, 
\begin{equation}\label{e:F3}
F_3\left(s, u, \nu; x\right)
= \frac{1}{2\pi i} \int_{(\sigma_v)} 
\frac{\Gamma\left(u-v\right) \Gamma\left(1-s-\nu+v\right) \Gamma\left(1-s+\nu+v\right)}
{\Gamma\left(1+u+v\right)} x^v \; dv, 
\end{equation}
where $-1+\Re(s) < \sigma_v < \Re(u)$. 
 We have 
\begin{equation}\label{e:F3_bound}
F_3(s, u, v; x) \ll x^{-\Re(u)}. 
\end{equation}
\end{corollary}

\begin{remark}\label{rmk:F1}
In Corollaries~\ref{cor:F1} and \ref{cor:F2}, we choose particular contours for integrals when the interval is not empty. 
If the interval is empty, we can instead choose any other contour which separates the poles of gamma functions in the given integral. 
\end{remark}

We now recall the definition of $F_2(s, u, k; x)$ given in \eqref{e:F2_def}. 
Fix $x>0$. 
For $s, u\in \C$, $\frac{1}{2}< \Re(s) < \Re(u)$ and $k\geq 1$, 
\begin{equation}
F_2(s, u, k; x) = -i^k \frac{1}{2\pi i} \int_{(-\Re(s)-\epsilon)}  
\frac{\Gamma\left(u+v\right)}{\Gamma\left(1+u-v\right)} 
\frac{\Gamma\left(1-s-v+\frac{k-1}{2}\right)}{\Gamma\left(s+v+\frac{k-1}{2}\right)} x^v \; dv. 
\end{equation}

\begin{lemma}\label{lem:F2}
We give the following explicit descriptions for $F_2(s, u, k; x )$ for $x>1$, $x=1$ and $x<1$. 

For $x>1$, 
\begin{multline}\label{e:F2_x>1_hyp}
F_2(s, u, k; x)
= -i^k \frac{\Gamma\left(1-s+u+\frac{k-1}{2}\right)}{\Gamma\left(s-u+\frac{k-1}{2}\right) \Gamma\left(1+2u\right)}
x^{-u} \;_2F_1\left({1-s+u+\frac{k-1}{2}, 1-s+u-\frac{k-1}{2}\atop 1+2u}; x^{-1}\right).
\end{multline}
When $x=1$, 
\begin{equation}\label{e:F2_x=1}
F_2(s, u, k; 1)
= -i^k \Gamma\left(2s-1\right) \frac{\sin\left(\pi\left(s+u-\frac{k-1}{2}\right)\right)}{\pi}
\frac{\Gamma\left(1-s-u+\frac{k-1}{2}\right)\Gamma\left(1-s+u+\frac{k-1}{2}\right) }
{\Gamma\left(s+u+\frac{k-1}{2}\right)\Gamma\left(s-u+\frac{k-1}{2}\right)}.
\end{equation}
When $0<x<1$, 
\begin{equation}\label{e:F2_x<1_hyp}
F_2(s, u, k; x) 
= -i^k 
\frac{\Gamma\left(1-s+u+\frac{k-1}{2}\right)}{\Gamma\left(s+u-\frac{k-1}{2}\right) \Gamma\left(k\right)}
x^{1-s+\frac{k-1}{2}} 
\;_2F_1\left({1-s+u+\frac{k-1}{2}, 1-s-u+\frac{k-1}{2}\atop k}; x\right).
\end{equation}
Since the hypergeometric function $\;_2F_1({a, b\atop c}; x)$ with any fixed branch 
(provided that $x=0, 1$ and $\infty$ are excluded) is a meromorphic function of $a, b$ and $c$ except for possible poles at 
$c\in \Z_{\leq 0}$, 
$F_2(s, u, k; x)$ also has meromorphic continuation to $s, u\in \C$. 
\end{lemma}

\begin{corollary}\label{cor:F2}
Assume that $\Re(s)\geq \frac{1}{2}$ and $\Re(s) < \Re(u)$. 
We have the following integral representations of $F_2$. 

For $x>1$, 
\begin{multline}\label{e:F2_x>1}
F_2(s, u, k; x)
= -i^k 
\frac{\sin\left(\pi\left(s-u+\frac{k-1}{2}\right)\right)}{\pi} 
\frac{\Gamma\left(1-s+u-\frac{k-1}{2}\right)}{\Gamma\left(s+u+ \frac{k-1}{2}\right)}
\\ \times \left(\frac{x}{x-1}\right)^{1-s+ \frac{k-1}{2}}
\frac{1}{2\pi i } 
\int_{(\sigma_v)} \frac{\Gamma\left(1-s+ \frac{k-1}{2}+v\right) \Gamma\left(s+ \frac{k-1}{2}+v\right) \Gamma\left(u-v\right)}{\Gamma\left(1+u+v\right)}
(x-1)^{-v} \; dv
\end{multline}
for $-1+\Re(s)- \frac{k-1}{2} < \sigma_v < \Re(u)$. 
 We have 
\begin{equation}\label{e:F2_x>1_bound}
\left(\frac{x}{x-1}\right)^{-1+s-\frac{k-1}{2}} F_2(s, u, k; x) \ll (x-1)^{-\Re(u)}. 
\end{equation}

%
%
When $0< x<1$, 
\begin{multline}\label{e:F2_x<1}
F_2(s, u, k; x) 
= -i^k  \frac{\sin(\pi(s+u-\frac{k-1}{2}))}{\pi} 
\frac{1}{\Gamma\left(s-u+\frac{k-1}{2}\right)\Gamma\left(s+u+\frac{k-1}{2}\right)}
\\ \times \left(\frac{x}{1-x}\right)^{1-s+\frac{k-1}{2}}
\frac{1}{2\pi i} \int_{(\sigma_v)} \Gamma\left(u-v\right) \Gamma\left(-u-v\right) 
\Gamma\left(s+\frac{k-1}{2}+v\right) \Gamma\left(1-s+\frac{k-1}{2}+v\right)
(1-x)^{-v} dv, 
\end{multline}
where $-1+\Re(s)-\frac{k-1}{2} < \sigma_v < -|\Re(u)|$. 
\end{corollary}

\begin{proof}[Proof of Lemma~\ref{lem:F2} and Corollary~\ref{cor:F2}]
Assume that $x>1$. 
Recalling the definition of $F_2(s, u, k; x)$ \eqref{e:F2_def}, by the reflection formula, 
\begin{multline}
F_2(s, u, k; x) 
= -i^k \frac{1}{2\pi i} \int_{(-\Re(s)-\epsilon)}  
\frac{\Gamma\left(u+v\right)}{\Gamma\left(1+u-v\right)} 
\frac{\Gamma\left(1-s-v+\frac{k-1}{2}\right)\Gamma\left(1-s-v-\frac{k-1}{2}\right)}
{\Gamma\left(s+v+\frac{k-1}{2}\right)\Gamma\left(1-s-v-\frac{k-1}{2}\right)} x^v \; dv
\\ = -i^k \frac{1}{2\pi i} \int_{(-\Re(s)-\epsilon)}  
\frac{\Gamma\left(u+v\right)\Gamma\left(1-s-v+\frac{k-1}{2}\right)\Gamma\left(1-s-v-\frac{k-1}{2}\right)}
{\Gamma\left(1+u-v\right)} 
\frac{\sin\left(\pi\left(s+v+\frac{k-1}{2}\right)\right)}{\pi} x^v \; dv.
\end{multline}
We move the $v$ line of integration to $\Re(v)=-L$ and take $L\to +\infty$. 
Since $x>1$, on the line of integration, $x^{\Re(v)} = x^{-L} \to 0$ as $L\to +\infty$. 
The integral converges absolutely for $\Re(s)>\frac{1}{2}$. 
Thus by collecting the residues at $v=-u-n$ for $n\geq 0$, we get
\begin{multline}
F_2(s, u, k; x)
\\ = -i^k \sum_{n=0}^\infty \frac{(-1)^n}{n!} 
\frac{\Gamma\left(1-s+u+\frac{k-1}{2}+n\right) \Gamma\left(1-s+u-\frac{k-1}{2}+n\right)}{\Gamma\left(1+2u+n\right)}
\frac{\sin\left(\pi\left(s-u+\frac{k-1}{2}-n\right)\right)}{\pi} x^{-u-n}.
\end{multline}
For $\sin\left(\pi\left(s-u+\frac{k-1}{2}-n\right)\right)=(-1)^n \sin\left(\pi\left(s-u+\frac{k-1}{2}\right)\right)$, 
and by the definition of the hypergeometric function via the Gauss series (since $|x|^{-1} <1$ the series converges absolutely), we get 
\begin{multline}
F_2(s, u, k; x)
\\ = -i^k\frac{\sin\left(\pi\left(s-u+\frac{k-1}{2}\right)\right)}{\pi} x^{-u}
\sum_{n=0}^\infty 
\frac{\Gamma\left(1-s+u+\frac{k-1}{2}+n\right) \Gamma\left(1-s+u-\frac{k-1}{2}+n\right)}{\Gamma\left(1+2u+n\right) n!}
x^{-n}
\\ = -i^k\frac{\sin\left(\pi\left(s-u+\frac{k-1}{2}\right)\right)}{\pi} x^{-u}
\frac{\Gamma\left(1-s+u+\frac{k-1}{2}\right) \Gamma\left(1-s+u-\frac{k-1}{2}\right)}{\Gamma\left(1+2u\right)}
\\ \times \;_2F_1\left({1-s+u+\frac{k-1}{2}, 1-s+u-\frac{k-1}{2}\atop 1+2u} ; x^{-1}\right).
\end{multline}
Then we get \eqref{e:F2_x>1_hyp}. 
Moreover, taking $x\to 1$ from $x>1$, since $\Re(s)>\frac{1}{2}$, by the Gauss formula, 
\begin{equation}
F_2(s, u, k; 1)
= -i^k\frac{\sin\left(\pi\left(s-u+\frac{k-1}{2}\right)\right)}{\pi} 
\frac{\Gamma\left(-1+2s\right)\Gamma\left(1-s+u+\frac{k-1}{2}\right) \Gamma\left(1-s+u-\frac{k-1}{2}\right)}
{\Gamma\left(s+u-\frac{k-1}{2}\right) \Gamma\left(s+u+\frac{k-1}{2}\right)}.
\end{equation}
By applying the reflection formulas, we get \eqref{e:F2_x=1}. 

By applying the Pfaff transformation, 
\begin{multline}
\;_2F_1\left({1-s+u+\frac{k-1}{2}, 1-s+u-\frac{k-1}{2}\atop 1+2u} ; x^{-1}\right)
\\ = (1-x^{-1})^{-1+s-u\mp \frac{k-1}{2}} \;_2F_1\left({1-s+u\pm \frac{k-1}{2}, s+u\pm \frac{k-1}{2}\atop 1+2u}; \frac{x^{-1}}{x^{-1}-1}\right).
\end{multline}
We choose one of the signs and get 
\begin{multline}\label{e:F2_x>1_inter}
F_2(s, u, k; x)
= -i^k\frac{\sin\left(\pi\left(s-u+\frac{k-1}{2}\right)\right)}{\pi} 
\frac{\Gamma\left(1-s+u+\frac{k-1}{2}\right) \Gamma\left(1-s+u-\frac{k-1}{2}\right)}{\Gamma\left(1+2u\right)}
\\ \times \left(\frac{x}{x-1}\right)^{1-s+ \frac{k-1}{2}} (x-1)^{-u} 
\;_2F_1\left({1-s+u+ \frac{k-1}{2}, s+u+ \frac{k-1}{2}\atop 1+2u};-(x-1)^{-1}\right).
\end{multline}
Recall the Barnes integral representation for $\;_2F_1$: for $z>0$, 
\begin{equation}
\;_2F_1\left(\left.a, b\atop c\right.; -z\right)
= \frac{\Gamma\left(c\right)}{\Gamma\left(a\right)\Gamma\left(b\right)}
\frac{1}{2\pi i} \int_C \frac{\Gamma\left(a+v\right)\Gamma\left(b+v\right)\Gamma\left(-v\right)}
{\Gamma\left(c+v\right)} z^v\;dv.
\end{equation}
Here the contour $C$ is taken to separate the pole of gamma functions. 
For $\frac{x^{-1}}{1-x^{-1}} = - (x-1)^{-1} <0$, we get
\begin{multline}
 \;_2F_1\left({1-s+u+ \frac{k-1}{2}, s+u+ \frac{k-1}{2}\atop 1+2u}; -(x-1)^{-1} \right)
 = \frac{\Gamma\left(1+2u\right)}{\Gamma\left(1-s+u+ \frac{k-1}{2}\right) \Gamma\left(s+u+ \frac{k-1}{2}\right)}
\\ \times \frac{1}{2\pi i } 
\int_{(\sigma_0)} \frac{\Gamma\left(1-s+u+ \frac{k-1}{2}+v\right) \Gamma\left(s+u+ \frac{k-1}{2}+v\right) \Gamma\left(-v\right)}{\Gamma\left(1+2u+v\right)}
(x-1)^{-v} \; dv
\end{multline}
for $\max\{-1+\Re(s), -\Re(s)\}-\Re(u)-\frac{k-1}{2} < \sigma_0 < 0$. 
Changing the variable $u+v$ to $v$, combining with \eqref{e:F2_x>1_inter}, we get \eqref{e:F2_x>1}. 

Now assume that $0<x<1$. 
Recalling the definition of $F_2(s, u, k; x)$ in \eqref{e:F2_def} again, by the reflection formula, 
\begin{equation}
F_2(s, u, k; x) = -i^k \frac{1}{2\pi i} \int_{(-\Re(s)-\epsilon)}  
\frac{\sin(\pi(-u+v))}{\pi}
\Gamma\left(u+v\right)\Gamma\left(-u+v\right) 
\frac{\Gamma\left(1-s-v+\frac{k-1}{2}\right)}{\Gamma\left(s+v+\frac{k-1}{2}\right)} x^v \; dv.
\end{equation}
For $x<1$, we now move the $v$ line of integration to $\Re(v)=L$ and taking the limit $L\to +\infty$. 
Similar to the above case when $x>1$, 
by collecting the residues at $v = 1-s+\frac{k-1}{2}+n$ for $n\geq 0$, 
and then applying the definition of the hypergeometric function via the Gauss series, we get

\begin{multline}\label{e:F2_x<1_inter}
F_2(s, u, k; x) 
= -i^k  \frac{\sin(\pi(s+u-\frac{k-1}{2}))}{\pi} x^{1-s+\frac{k-1}{2}}
\frac{\Gamma\left(1-s+u+\frac{k-1}{2}\right) \Gamma\left(1-s-u+\frac{k-1}{2}\right)}{\Gamma\left(k\right)}
\\ \times \;_2F_1\left({1-s+u+\frac{k-1}{2}, 1-s-u+\frac{k-1}{2}\atop k}; x\right).
\end{multline}
 By applying the reflection formula for gamma functions, we get \eqref{e:F2_x<1_hyp}.

By \cite[15.6.7]{DLMF}, for $1-x>0$, 
\begin{multline}
\frac{1}{\Gamma\left(k\right)} \;_2F_1\left({1-s+u+\frac{k-1}{2}, 1-s-u+\frac{k-1}{2}\atop k-1}; x\right)
\\ = \frac{1}{\Gamma\left(1-s+u+\frac{k-1}{2}\right)\Gamma\left(1-s-u+\frac{k-1}{2}\right)
\Gamma\left(s-u+\frac{k-1}{2}\right)\Gamma\left(s+u+\frac{k-1}{2}\right)}
\\ \times \frac{1}{2\pi i} \int_{(\sigma_0)} \Gamma\left(1-s+u+\frac{k-1}{2}+v\right) \Gamma\left(1-s-u+\frac{k-1}{2}+v\right) 
\Gamma\left(-1+2s-v\right) \Gamma\left(-v\right)
(1-x)^{v} dv, 
\end{multline}
where $-1+\Re(s)+|\Re(u)|-\frac{k-1}{2}< \sigma_0<0$, for $\Re(s)\geq \frac{1}{2}$. 
As explained in Remark~\ref{rmk:F1}, if the interval is empty, we can instead choose any other contour which separates the poles of gamma functions. 
By changing the variable $1-s+\frac{k-1}{2}+v$ to $-v$ and applying to \eqref{e:F2_x<1_inter}, we get \eqref{e:F2_x<1}.
\end{proof}

\section{Proof of Theorem~\ref{thm:first_holo} and Theorem~\ref{thm:first_nonholo}}

Take $\phi$ as a holomorphic cusp form of weight $k$ or a non-holomorphic automorphic form of type $\nu\in i\R$, as described in \S\ref{ss:notation}. 
Take $\sigma_u$ such that $\frac{1}{2}\leq \Re(s) < \sigma_u < \frac{3}{2}$. 
Via Proposition~\ref{prop:Lq_nonholo}, we define
 \begin{equation}\label{e:Phi_def}
\Phi(s, \phi; n) = L^{(M)}(2s, \xi) 
\frac{4}{2\pi i} \int_{(\sigma_u)}\frac{h(u/i) u}{\cos(\pi u)} \sum_{q\equiv 0\bmod{M}} \Phi(s, u; \phi; n) \; du
\end{equation}
where $\Phi$ is $M^\pm$, $L_1^\pm$ or $L_2$.
When $\xi=1_N$, we let 
\begin{equation}\label{e:M0def}
M_0(s, \phi; n) = \zeta^{(M)}(2s) 
\frac{4}{2\pi i} \int_{(\sigma_u)}\frac{h(u/i) u}{\cos(\pi u)} \sum_{q\equiv 0\bmod{M}} M_q(s, u; \phi; n) \; du
\end{equation}
and 
\begin{equation}\label{e:Mdef}
M(s, \phi; n) = L^{(M)}(2s, \xi) \frac{C_\phi(n)}{n^s} H_0(h) + \delta_{\xi=1_N} M_0(s, \phi; n).
\end{equation}
Then, by \eqref{e:K_open} and \eqref{e:K0_after_interchange} and \eqref{e:Lq_prop} in Proposition~\ref{prop:Lq_nonholo}, 
\begin{multline}
K(s, \phi; n, h)
= M(s, \phi; n) + M^+(s, \phi; n) + M^-(s, \phi; n)
\\ + L_{1}^+(s, \phi; n) + L_1^-(s, \phi; n) + L_2(s, \phi; n). 
\end{multline}

A necessary ingredient for the analysis of $M^\pm(s, \phi; n)$, $L_1^\pm(s, \phi; n)$ and $L_2(s, \phi; n)$ is the following lemma. 
\begin{lemma}\label{lem:sum_cxiq}
Let $\xi$ be a Dirichlet character modulo $N$, which is induced from a primitive character $\xi_*$ of conductor $R\mid N$.
Fix a positive integer $M$ which is divisible by $N$. 
For a non-zero integer $n$ divisible by $\frac{M}{R\prod_{p\mid M, p\nmid R} p}$, 
\begin{equation}\label{e:sum_cxiq}
L^{(M)}(2s, \xi_*) \sum_{\substack{q\geq 1, \\ M\mid q}} \frac{c_{\xi|q}(n)}{q^{2s}}
= R^{-2s} \xi_*(\sgn(n)) \tau(\xi_*) \sigma_{1-2s}(n; \xi_*, M).
\end{equation}
Otherwise the series is zero. 
Here when $n$ is divisible by $\frac{M}{R\prod_{p\mid M, p\nmid R} p}$, 
\begin{equation}
\sigma_{1-2s}(n; \xi_*, M)
=
\sigma_{1-2s}(n; \xi_*\cdot 1_M) P_M(2s-1, n; \xi_*)
\prod_{p\nmid M} \overline{\xi_*(p^{\ord_p(n)})}
\prod_{p\mid M} p^{(1-2s)\ord_p(n)}
\end{equation}
and $\sigma_{1-2s}(n; \xi_*, M)=0$ otherwise, where 
\begin{equation}
\sigma_{1-2s}(n; \chi) = \sum_{d\mid n} d^{1-2s}\chi(d) 
\end{equation}
for a Dirichlet character $\chi$, 
and 
\begin{equation}\label{e:PM}
P_M(u, n; \xi_*)
= \prod_{\substack{p\mid M,\\ p\nmid R}} 
\bigg\{-p^{-1-u} \xi_*(p) + (1-p^{-1}) \sigma_{u}(p^{\ord_p(n)-\ord_p(M)}; \overline{\xi_*})\bigg\}. 
\end{equation}
\end{lemma}

\begin{lemma}\label{lem:M_nonholo}
For $\Re(s) \geq \frac{1}{2}$, we have 
\begin{multline}\label{e:M+}
M^\pm (s, \phi; n)
= - 2c_\phi^\pm \frac{\pi^{\frac{1}{2}\pm \nu}(2\pi)^{2s-2\mp2\nu} }{\Gamma\left(\frac{1}{2}\pm\nu\right)}
R^{-2s}\tau(\xi_*) \sigma_{1-2s}(n; \xi_*, M) n^{s-1\mp\nu} 
\\ \times  \frac{\cos(\pi(s\mp\nu))}{\pi} \frac{2}{\pi} \int_{-\infty}^\infty 
h(r)r \tanh(\pi r) \Gamma\left(1-s\pm\nu+ir\right) \Gamma\left(1-s\pm\nu-ir\right)\; dr. 
\end{multline}
\end{lemma}
\begin{proof}
Applying \eqref{e:Mq+} and \eqref{e:sum_cxiq} to \eqref{e:Phi_def}, 
and since $L^{(M)}(2s, \xi_*) = L^{(M)}(2s, \xi)$ for $R\mid M$
\begin{multline}
M^+(s, \phi; n)
= 2c_\phi^+\frac{\pi^{\frac{1}{2}+\nu}(2\pi)^{2s-2-2\nu} }{\Gamma\left(\frac{1}{2}+\nu\right)}
n^{s-1-\nu} 
R^{-2s}\tau(\xi_*) \sigma_{1-2s}(n; \xi_*, M)
\\ \times \frac{4}{2\pi i} \int_{(\sigma_u)}\frac{h(u/i) u}{\cos(\pi u)} 
\frac{\Gamma\left(u-s+1+\nu\right)}{\Gamma\left(u+s-\nu\right)} \; du.
\end{multline}
Here $\frac{1}{2}\leq \Re(s) < \sigma_u < \frac{3}{2}$. 
Without passing a pole, we move the contour for the $u$-integral to $\Re(u)=0$. 
By the reflection formula, 
\begin{equation}
\frac{1}{\Gamma\left(u+s-\nu\right)} 
= \frac{\Gamma\left(-u-s+1+\nu\right) \big(\cos(\pi(s-\nu)) \sin(\pi u) + \sin(\pi(s-\nu))\cos(\pi u)\big)}{\pi}.
\end{equation}
Since $h(-u/i)=h(u/i)$, we get
\begin{multline}
\frac{4}{2\pi i} \int_{(0)}\frac{h(u/i) u}{\cos(\pi u)} 
\frac{\Gamma\left(u-s+1+\nu\right)}{\Gamma\left(u+s-\nu\right)} \; du
\\ = \frac{\sin(\pi(s-\nu))}{\pi} \frac{4}{2\pi i} \int_{(0)}
h(u/i) u \tan(\pi u) \Gamma\left(u-s+1+\nu\right) \Gamma\left(-u-s+1+\nu\right) \; du
\end{multline}
Writing $u=ir$, 
we obtain \eqref{e:M+} for $M^+$. 

The corresponding $M^-(s, \phi; n)$ is obtained in the same way. 
\end{proof}

In the following lemmas we obtain formulas for $M_0$, $L_1^\pm$ and $L_2$. 
The proofs of the lemmas are similar. 
First we apply the formulas given in Proposition~\ref{prop:Lq_nonholo} 
 to \eqref{e:Phi_def} for $\Phi$ is $L_1^\pm$ and $L_2$, and \eqref{e:M0def} for $M_0$. 
Then we apply Lemma~\ref{lem:sum_cxiq}.

\begin{lemma}\label{lem:M0_nonholo}
Assume that $\xi$ is the trivial character modulo $N$. 
When $\phi$ is a holomorphic cusp form of weight $k$, then $k$ should be even, for $\Re(s)\geq \frac{1}{2}$, we get 
\begin{equation}\label{e:M0_holo}
M_0(s, \phi; n)
= \zeta(2-2s) (2\pi)^{4s-2} 
M^{1-2s} \prod_{p\mid M} (1-p^{-1})
\frac{C_\phi(n)}{n^{1-s}} H_{\frac{k-1}{2}}(s; h).
\end{equation}
Here $H_{\frac{k-1}{2}}(s; h)$ is given in \eqref{e:Hk-1/2_h}. 
When $\phi$ is a non-holomorphic automorphic form of type $\nu$, 
\begin{equation}\label{e:M0_nonholo}
M_0(s, \phi; n) 
= -\zeta(2-2s) (2\pi)^{4s-3} 
M^{1-2s} \prod_{p\mid M}(1-p^{-1}) \frac{C_\phi(n)}{n^{1-s}} \frac{\pi }{\sin(\pi s) }H_\nu(s; h).
\end{equation}
Here $H_\nu(s; h)$ is given in \eqref{e:Hnu_h} when $\nu\notin \frac{1}{2}\Z$. 
\end{lemma}
\begin{proof}
We first note that 
\begin{equation}\label{e:sum_varphi}
\sum_{\substack{q\geq 1, \\ q\equiv 0\bmod{M}}} \frac{\varphi(q)}{q^{2s}}
= \frac{\zeta(2s-1)}{\zeta^{(M)}(2s)} M^{1-2s} \prod_{p\mid M} (1-p^{-1}). 
\end{equation}

When $\phi$ is a holomorphic cusp form, 
recalling \eqref{e:M0def} and \eqref{e:Mq_holo}, and by \eqref{e:sum_varphi}, 
\begin{multline}
M_0(s, \phi; n)
= 
(2\pi)^{2s-1} i^k
\Gamma\left(2s-1\right)
\frac{C_\phi(n)}{n^{1-s}} 
\zeta(2s-1) M^{1-2s}\prod_{p\mid M}(1-p^{-1})
\\ \times 
\frac{4}{2\pi i}\int_{(\sigma_u)} 
\frac{h(u/i) u}{\cos(\pi u)} 
\frac{\sin\left(\pi\left(s+u-\frac{k-1}{2}\right)\right)}{\pi}
\frac{\Gamma\left(1-s-u+\frac{k-1}{2}\right)\Gamma\left(1-s+u+\frac{k-1}{2}\right) }
{\Gamma\left(s+u+\frac{k-1}{2}\right)\Gamma\left(s-u+\frac{k-1}{2}\right)}
\; du. 
\end{multline}
Here $\frac{1}{2}\leq \Re(s) < \sigma_u < \frac{3}{2}$. 
We first move the contour of the $u$-integral to $\Re(u)=0$. 
Note that $(-1)^k=\xi(-1)=1$, i.e., $k$ must be even. 
So 
\begin{equation}
\sin\left(\pi\left(s+u-\frac{k-1}{2}\right)\right)
= (-1)^{\frac{k}{2}} \cos(\pi(s+u)) 
= (-1)^{\frac{k}{2}} \big(\cos(\pi s) \cos(\pi u)-\sin(\pi s) \sin(\pi u)\big). 
\end{equation}
For $h(-u/i) = h(u/i)$, we get
\begin{multline}
\frac{4}{2\pi i}\int_{(0)} 
\frac{h(u/i) u}{\cos(\pi u)} 
\frac{\sin\left(\pi\left(s+u-\frac{k-1}{2}\right)\right)}{\pi}
\frac{\Gamma\left(1-s-u+\frac{k-1}{2}\right)\Gamma\left(1-s+u+\frac{k-1}{2}\right) }
{\Gamma\left(s+u+\frac{k-1}{2}\right)\Gamma\left(s-u+\frac{k-1}{2}\right)}
\; du
\\ = - (-1)^{\frac{k}{2}} \frac{\sin(\pi s)}{\pi}\frac{4}{2\pi i}\int_{(0)} 
h(u/i) u \tan(\pi u)
\frac{\Gamma\left(1-s-u+\frac{k-1}{2}\right)\Gamma\left(1-s+u+\frac{k-1}{2}\right) }
{\Gamma\left(s+u+\frac{k-1}{2}\right)\Gamma\left(s-u+\frac{k-1}{2}\right)}
\; du.
\end{multline}
Writing $u=ir$, we have 
\begin{equation}
= \frac{i^k \sin\left(\pi s\right)}{\pi} 
\frac{4}{2\pi}  \int_{-\infty}^\infty 
h(r) r \tanh(\pi r)
\frac{\Gamma\left(1-s+\frac{k-1}{2}+ir\right)\Gamma\left(1-s+\frac{k-1}{2}-ir\right)}
{\Gamma\left(s+\frac{k-1}{2}+ir\right)\Gamma\left(s+\frac{k-1}{2}-ir\right)}\; dr.
\end{equation}
Applying the functional equation of the Riemann zeta function and the Legendre duplication formula 
for gamma functions to the remaining part, gives us \eqref{e:M0_holo}.


When $\phi$ is a Maass form of type $\nu$, 
recalling \eqref{e:M0def} and \eqref{e:Mq_nonholo}, and by \eqref{e:sum_varphi}, 
\begin{multline}
M_0(s, \phi; n) 
=  -(2\pi)^{2s-1} \zeta(2s-1)M^{1-2s} \prod_{p\mid M}(1-p^{-1}) \frac{C_\phi(n)}{n^{1-s}}
\\ \times \frac{4}{2\pi i} \int_{(\sigma_u)}\frac{h(u/i) u}{\cos(\pi u)} 
\frac{\Gamma\left(1-s-\nu+u\right) \Gamma\left(1-s+\nu+u\right)}
{\Gamma\left(\frac{1}{2}-s+u\right) \Gamma\left(\frac{1}{2}+s-u\right)}
\frac{\Gamma\left(2s-1\right)}
{\Gamma\left(s+\nu+u\right)\Gamma\left(s-\nu+u\right)}
\; du.
\end{multline}
By applying the reflection formula, 
\begin{multline}
\frac{4}{2\pi i} \int_{(\sigma_u)}\frac{h(u/i) u}{\cos(\pi u)} 
\frac{\Gamma\left(1-s-\nu+u\right) \Gamma\left(1-s+\nu+u\right)}
{\Gamma\left(\frac{1}{2}-s+u\right) \Gamma\left(\frac{1}{2}+s-u\right)}
\frac{\Gamma\left(2s-1\right)}{\Gamma\left(s+\nu+u\right)\Gamma\left(s-\nu+u\right)}\; du
\\ = \Gamma\left(2s-1\right)\frac{1}{\pi}
\frac{4}{2\pi i} \int_{(\sigma_u)}h(u/i) u
\frac{\cos(\pi(s-u))}{\cos(\pi u)} 
\frac{\Gamma\left(1-s-\nu+u\right) \Gamma\left(1-s+\nu+u\right)}
{\Gamma\left(s+\nu+u\right)\Gamma\left(s-\nu+u\right)}
\; du.
\end{multline}
Applying the functional equation of the Riemann zeta function and the Legendre duplication formula for gamma functions, gives us \eqref{e:M0_nonholo}.
\end{proof}

Recalling \eqref{e:Mdef}, when $\xi$ is the trivial character mod $N$, 
\begin{equation}
M(s, \phi; n) = \zeta^{(M)}(2s) \frac{C_\phi(n)}{n^s} H_0(h) + M_0(s, \phi; n).
\end{equation}
The following lemma describes $M(1/2, \phi; n)$.

\begin{lemma}\label{lem:M_nonholo_t=0}
Assume that $\xi$ is the trivial character mod $N$. 
When $\phi$ is a holomorphic cusp form of weight $k$, taking $s=1/2$, 
\begin{multline}\label{e:M_holo_t=0}
M(1/2, \phi; n) 
= \frac{C_\phi(n)}{\sqrt{n}} \frac{\varphi(M)}{M} 
\frac{1}{\pi^2} \int_{-\infty}^\infty h(r) r \tanh(\pi r) 
\bigg(\psi\left(\frac{k}{2}+ir\right) + \psi\left(\frac{k}{2}-ir\right)\bigg)
\; dr
\\ + \frac{C_\phi(n)}{\sqrt{n}} \frac{\varphi(M)}{M} 
\bigg\{\sum_{p\mid M} \frac{\log p}{1-p^{-1}} + \log \left(\frac{M}{(2\pi)^2 n}\right) + 2\gamma_0 \bigg\} 
H_0(h)
\end{multline}
When $\phi$ is a non-holomorphic automorphic form of type $\nu\in i\R$, 
\begin{multline}\label{e:M_nonholo_t=0}
M(1/2, \phi; n) 
= \frac{C_\phi(n)}{\sqrt{n}} \frac{\varphi(M)}{M} 
\frac{1}{\pi^2} \int_{-\infty}^\infty h(r) r \tanh(\pi r) 
\\ \times 
\frac{1}{2} \bigg(\psi\left(\frac{1}{2}+\nu+ir\right) + \psi\left(\frac{1}{2}+\nu-ir\right)
+ \psi\left(\frac{1}{2}-\nu+ir\right) + \psi\left(\frac{1}{2}-\nu-ir\right)\bigg)
\; dr
\\ + \frac{C_\phi(n)}{\sqrt{n}} \frac{\varphi(M)}{M} 
\bigg\{\sum_{p\mid M} \frac{\log p}{1-p^{-1}} + \log \left(\frac{M}{(2\pi)^2 n}\right) + 2\gamma_0 \bigg\} 
H_0(h)
\end{multline}
Here $\psi(z) = \frac{\Gamma'}{\Gamma}(z)$ and $H_0(h)$ is given in \eqref{e:H0}. 
\end{lemma}

\begin{proof}
Recalling \eqref{e:M0_holo} and \eqref{e:M0_nonholo}, when $\xi=1_N$ and taking $s=1/2+it$, 
we write 
\begin{equation}
M(1/2+it, \phi; n)
= \zeta(1+2it) M_1(t) + \zeta(1-2it) M_2(t)
\end{equation}
where 
\begin{equation}
M_1 (t) = \frac{C_\phi(n)}{n^{\frac{1}{2}+it}} \prod_{p\mid M}(1-p^{-1-2it}) H_0(h), 
\end{equation}
\begin{equation}
M_2(t) = (2\pi)^{4it} M^{-2it} \frac{\varphi(M)}{M}
\frac{C_\phi(n)}{n^{\frac{1}{2}-it}} 
H_{\frac{k-1}{2}}(1/2+it; h)
\end{equation}
when $\phi$ is a holomorphic cusp form, and 
\begin{multline}
M_2(t) = (2\pi)^{4it} 
M^{-2it} \prod_{p\mid M} (1-p^{-1})
\frac{C_\phi(n)}{n^{\frac{1}{2}-it}} 
\\ \times \frac{1}{2} 
\bigg\{\frac{\sin(\pi(it+\nu))}{\sin(\pi \nu) \cos(\pi it)} H_{-\nu}(1/2+it; h)
-\frac{\sin(\pi(it-\nu))}{\cos(\pi it)\sin(\pi \nu)} H_{\nu}(1/2+it; h)
\bigg\}, 
\end{multline}
when $\phi$ is a Maass form. 
Here $H_\nu(s; h)$ is given in \eqref{e:Hnu_h}. 

We now assume that $\phi$ is holomorphic. 
The arguments are almost identical in both cases: when $\phi$ is holomorphic or non-holomorphic. 
By the Laurent series expansion of $\zeta(s)$ at $s=1$, we get
\begin{equation}
\zeta(1\pm 2it) = \pm \frac{1}{2it} + \gamma_0 + O(t)
\end{equation}
where $\gamma_0$ is the Euler-Mascheroni constant. 
Then 
\begin{equation}
M(1/2+it, \phi; n) = \frac{1}{2it} (M_1(0) -M_2(0)) + \gamma_0(M_1(0)+M_2(0)) + \frac{1}{2i} (M_1'(0) - M_2'(0)) + O(t).
\end{equation}

For 
\begin{equation}
M_1(0) 
= \frac{\varphi(M)}{M} \frac{C_\phi(n)}{\sqrt{n}} H_0(h)
= M_2(0), 
\end{equation}
\begin{equation}
M_1'(0) =i \bigg(2\sum_{p\mid M} \frac{\log p}{1-p^{-1}}- \log n\bigg)
\frac{\varphi(M)}{M}\frac{C_\phi(n)}{\sqrt{n}} H_0(h)
\end{equation}
and 
\begin{multline}
M_2'(0) = i \bigg( 4\log (2\pi) - 2 \log M+ \log n \bigg) \frac{\varphi(M)}{M} 
\frac{C_\phi(n)}{\sqrt{n}} H_0(h)
\\ -  \frac{\varphi(M)}{M} \frac{C_\phi(n)}{\sqrt{n}} \frac{2i}{\pi^2} \int_{-\infty}^\infty rh(r) \tanh(\pi r) 
\bigg(\psi\left(\frac{k}{2}+ir\right) + \psi\left(\frac{k}{2}-ir\right)\bigg)
\; dr, 
\end{multline}
taking $t=0$, we get \eqref{e:M_nonholo_t=0}.

%
\end{proof}

In the following lemmas we now present explicit expressions for $L_1^\pm(s, \phi; n)$. 

\begin{lemma}\label{lem:L1+}
When $\phi$ is either a non-holomorphic automorphic form of type $\nu\in i\R$ and weight $0$, 
or a holomorphic cusp form of weight $k$, $\nu=\frac{k-1}{2}$, 
for $\Re(s)\geq \frac{1}{2}$, 
\begin{multline}\label{e:L1+nonholo_holo}
L_{1}^+(s, \phi; n)
= - i^k (2\pi)^{2s-1} R^{-2s} \tau(\xi_*)\xi_*(-1)
\frac{4}{2\pi i} \int_{(\sigma_u)}\frac{h(u/i) u}{\cos(\pi u)} 
\frac{\cos\left(\pi\left(s-u+\frac{k}{2}\right)\right)}{\pi} 
\frac{\Gamma\left(1-s-\nu+u\right)}{\Gamma\left(s+\nu+u\right)}
\\ \times 
\frac{1}{2\pi i } 
\int_{(\sigma_v)} \frac{\Gamma\left(1-s+\nu+v\right) \Gamma\left(s+\nu+v\right) \Gamma\left(u-v\right)}{\Gamma\left(1+u+v\right)}
n^v D(v+1/2, s-1/2; \phi; n)
\; dv \; du. 
\end{multline}

Here we choose the contours $\Re(v)=\sigma_v$ and $\Re(u)=\sigma_u$ satisfying $\Re(s) < \sigma_v < \sigma_u<\frac{3}{2}$. 
For this interval the series and the integral converge absolutely. 
The shifted Dirichlet series $D(v+1/2, s-1/2; \phi; n)$ is given in \eqref{e:Dphi}. 
\end{lemma}

\begin{proof}
By applying \eqref{e:Lq1+nonholo} to \eqref{e:Phi_def}, and then by Lemma~\ref{lem:sum_cxiq},
\begin{multline}
L_{1}^+(s, \phi; n)
 = - (2\pi)^{2s-1} R^{-2s} \tau(\xi_*) \xi_*(-1)
\frac{4}{2\pi i} \int_{(\sigma_u)}\frac{h(u/i) u}{\cos(\pi u)} 
\\ \times \sum_{m=1}^\infty \frac{C_\phi(m+n)\sigma_{1-2s}(m; \xi_*, M)}{(m+n)^{1-s}} 
F_1\left(s, u, \nu; \frac{m+n}{n}\right)\; du.
\end{multline}
By using the integral representation \eqref{e:F1_x>1} for $F_1$, we get
\begin{multline}
L_{1}^+(s, \phi; n)
= - (2\pi)^{2s-1} R^{-2s} \tau(\xi_*)\xi_*(-1)
\frac{4}{2\pi i} \int_{(\sigma_u)}\frac{h(u/i) u}{\cos(\pi u)} \frac{\cos(\pi(s-u))}{\pi} 
\frac{\Gamma\left(1-s- \nu+u\right)}
{\Gamma\left(s+ \nu+u\right)}
\\ \times \frac{1}{2\pi i} \int_{(\sigma_v)} \frac{\Gamma\left(1-s+ \nu+v\right) \Gamma\left(s+ \nu+v\right) \Gamma\left(u-v\right)}
{\Gamma\left(1+u+v\right)} 
\\ \times n^v\sum_{m=1}^\infty \frac{C_\phi(m+n)(m+n)^{\nu} \sigma_{1-2s}(m; \xi_*, M)m^{-\frac{1}{2}+s}}{m^{\frac{1}{2}+v+\nu}} 
\; dv \; du, 
\end{multline}
where $\max\{-1+\Re(s), -\Re(s)\} < \sigma_v < \Re(u)$. 
Here $\sigma_{1-2s}(n; \xi_*, M)$ is defined in Lemma~\ref{lem:sum_cxiq}. 

Since $\Re(s)< \Re(u)$, $\Re(s) < \Re(v) < \Re(u)$, 
the intersection of the intervals $1< 1-\Re(s) +\Re(v)$ and $\max\{-1+\Re(s), -\Re(s)\} < \sigma_v < \Re(u)$, is non-empty.

When $\phi$ is a holomorphic cusp form of weight $k$, 
the proof is almost parallel to the non-holomorphic case.  
We apply \eqref{e:Lq1+holo} to \eqref{e:Phi_def}, and then by Lemma~\ref{lem:sum_cxiq}, 
and use the integral representation \eqref{e:F2_x>1} for $F_2$.

\end{proof}

\begin{lemma}\label{lem:L1-nonholo}
When $\phi$ is a non-holomorphic automorphic form of type $\nu\in i\R$, for $\Re(s)\geq 1/2$, 
\begin{multline}\label{e:L1-_withF1}
L_1^-(s, \phi; n) 
\\ = -(2\pi)^{2s-1} R^{-2s} \tau(\xi_*)
\frac{4}{2\pi i}\int_{(\sigma_u)} \frac{h(u/i)u}{\cos(\pi u)} 
\sum_{m=1}^{n-1} \frac{C_\phi(m) \sigma_{1-2s}(n-m; \xi_*, M)}{m^{1-s}} F_1\left(s, u, \nu; \frac{m}{n}\right)
\; du.
\end{multline}
When $\phi$ is a holomorphic cusp form of weight $k$, for $\Re(s)\geq \frac{1}{2}$, 
\begin{multline}\label{e:L1-_withF2}
L_1^-(s, \phi; n) 
\\ = -(2\pi)^{2s-1} R^{-2s} \tau(\xi_*)
\frac{4}{2\pi i}\int_{(\sigma_u)} \frac{h(u/i)u}{\cos(\pi u)} 
\sum_{m=1}^{n-1} \frac{C_\phi(m) \sigma_{1-2s}(n-m; \xi_*, M)}{m^{1-s}} F_2\left(s, u, k; \frac{m}{n}\right)
\; du.
\end{multline}

When $\Re(s)=\frac{1}{2}$, we have the following integral representations. 
When $\phi$ is a non-holomorphic automorphic form of type $\nu\in i\R$, 
\begin{multline}\label{e:L1-nonholo}
L_1^-(s, \phi; n) 
= (2\pi)^{2s-1} R^{-2s} \tau(\xi_*)
\sum_{\pm}
\frac{\cos(\pi(s\pm \nu))}{2\pi \sin(\pm \pi \nu)}
\frac{4}{2\pi i}\int_{(\sigma_u)} 
\frac{h(u/i)u \tan(\pi u)}{\Gamma\left(s\mp \nu-u\right) \Gamma\left(s\mp \nu+u\right)}
\\ \times 
\frac{1}{2\pi i} \int_{(\sigma_v)} \Gamma\left(u-v\right) \Gamma\left(-u-v\right) \Gamma\left(s\mp \nu+v\right) 
\Gamma\left(1-s\mp \nu+v\right) 
\\ \times n^v\sum_{m=1}^{n-1} \frac{C_\phi(n-m)(n-m)^{\mp \nu} \sigma_{1-2s}(m; \xi_*, M)m^{-\frac{1}{2}+s}}
{m^{\frac{1}{2}+v\mp\nu}} 
\; dv \; du.
\end{multline}
When $\phi$ is a holomorphic cusp form of weight $k$,
\begin{multline}\label{e:L-holo}
L_1^-(s, \phi; n) 
= i^k (2\pi)^{2s-1} R^{-2s} \tau(\xi_*)
\frac{\cos(\pi(s-\frac{k-1}{2}))}{\pi} 
\frac{4}{2\pi i}\int_{(\sigma_u)} 
\frac{h(u/i)u \tan(\pi u)}{\Gamma\left(s-u+\frac{k-1}{2}\right)\Gamma\left(s+u+\frac{k-1}{2}\right)}
\\ \times 
\frac{1}{2\pi i} \int_{(\sigma_v)} \Gamma\left(u-v\right) \Gamma\left(-u-v\right) 
\Gamma\left(s+\frac{k-1}{2}+v\right) \Gamma\left(1-s+\frac{k-1}{2}+v\right)
\\ \times n^v \sum_{m=1}^{n-1} \frac{c_\phi(n-m) \sigma_{1-2s}(m; \xi_*, M)m^{-\frac{1}{2}+s}}
{m^{\frac{1}{2}+v+\frac{k-1}{2}}} 
\; dv \; du. 
\end{multline}
Recall that $c_\phi(m) = m^{\frac{k-1}{2}} C_\phi(m)$. 
Here we choose the contours $\Re(u)=\sigma_u$ and $\Re(v)=\sigma_v$, 
where $0< \sigma_u< \frac{1}{2}$ and $-\frac{1}{2}< \sigma_v<-\sigma_u$. 
\end{lemma}

\begin{proof}
By applying \eqref{e:Lq1-nonholo} to \eqref{e:Phi_def}, and then by Lemma~\ref{lem:sum_cxiq}, 
we get \eqref{e:L1-_withF1} and \eqref{e:L1-_withF2}. 

By using the integral representation \eqref{e:F1_x<1} for $F_1$, we get
\begin{multline}
L_1^-(s, \phi; n) 
= (2\pi)^{2s-1} R^{-2s} \tau(\xi_*)
\sum_{\pm} \frac{4}{2\pi i}\int_{(\sigma_u)} \frac{h(u/i)u}{\cos(\pi u)} 
\frac{\sin(\pi(s\pm \nu+u))}{2\pi \sin(\pm \pi \nu)}
\frac{1}{\Gamma\left(s\mp \nu-u\right) \Gamma\left(s\mp \nu+u\right)}
\\ \times 
\frac{1}{2\pi i} \int_{C} \Gamma\left(u-v\right) \Gamma\left(-u-v\right) \Gamma\left(s\mp\nu+v\right) 
\Gamma\left(1-s\mp\nu+v\right) 
\\ \times n^v\sum_{m=1}^{n-1} \frac{C_\phi(m)m^{\mp \nu} \sigma_{1-2s}(n-m; \xi_*, M)(n-m)^{-\frac{1}{2}+s}}{(n-m)^{\frac{1}{2}+v\mp\nu}} 
\; dv \; du
\end{multline}
Here the contour $C$ separates the poles of gamma functions. 

Initially we take the contour for the $u$-line of integration as $\Re(u)=\sigma_u$, 
$\frac{1}{2} \leq \Re(s) < \sigma_u < \frac{3}{2}$. 
Assume that $\Re(s)=1/2$ and move the $u$-line of integration $\Re(u)=\sigma_u$ where $0<\sigma_u< \frac{1}{2}$, 
and then take the contour $C$ as $\Re(v)=\sigma_v$ where $-\frac{1}{2} < \sigma_v < -\Re(u)$. 
On these contour lines of integrals, 
using that $\sin(\pi(s+\nu+u)) = \sin(\pi(s+\nu))\cos(\pi u) + \cos(\pi(s+\nu))\sin(\pi u)$, 
we get
\begin{multline}
\frac{4}{2\pi i}\int_{(\sigma_u)} \frac{h(u/i)u}{\cos(\pi u)} 
\frac{\sin(\pi(s+\nu+u))}{2\pi \sin(\pi \nu)}
\frac{\Gamma\left(u-v\right) \Gamma\left(-u-v\right)}{\Gamma\left(s-\nu-u\right) \Gamma\left(s-\nu+u\right)}\; du
\\ = \frac{\sin(\pi(s+\nu))}{2\pi \sin(\pi \nu)}
\frac{4}{2\pi i}\int_{(\sigma_u)} h(u/i)u
\frac{\Gamma\left(u-v\right) \Gamma\left(-u-v\right)}{\Gamma\left(s-\nu-u\right) \Gamma\left(s-\nu+u\right)}\; du
\\ + \frac{\cos(\pi(s+\nu))}{2\pi \sin(\pi \nu)}
\frac{4}{2\pi i}\int_{(\sigma_u)} h(u/i)u \tan(\pi u)
\frac{\Gamma\left(u-v\right) \Gamma\left(-u-v\right)}{\Gamma\left(s-\nu-u\right) \Gamma\left(s-\nu+u\right)}\; du.
\end{multline}
Now we compute the first integral. 
Since $\Re(v) < -\Re(u)$, and also $h(u/i)=h(-u/i)$, by moving the $u$-line of integration to $\Re(u)=-\sigma_u$, 
and then changing the variable, we get 
\begin{equation}\label{e:hu_evenuint_0}
\frac{1}{2\pi i}\int_{(\sigma_u)} h(u/i)u
\frac{\Gamma\left(u-v\right) \Gamma\left(-u-v\right)}{\Gamma\left(s-\nu-u\right) \Gamma\left(s-\nu+u\right)}\; du
=0.
\end{equation}
Then we get \eqref{e:L1-nonholo}. 

When $\phi$ is a holomorphic cusp form of weight $k$, 
the proof is almost identical to the proof for the non-holomorphic case. 
We apply \eqref{e:Lq1-holo} to \eqref{e:Phi_def}, and then apply Lemma~\ref{lem:sum_cxiq}. 
Finally we use the integral representation \eqref{e:F2_x<1} for $F_2$. 
The remaining computations are also similar to the non-holomorphic case. 

\end{proof}

\begin{lemma}\label{lem:L2}
For $\Re(s)\geq \frac{1}{2}$, when $\phi$ is a Maass form of type $\nu$, 
\begin{multline}\label{e:L2}
L_2 (s, \phi; n)
= (2\pi)^{2s-1} R^{-2s} \tau(\xi_*) c_\phi(-1)\frac{\cos(\pi \nu)}{\pi} 
\frac{4}{2\pi i} \int_{(\sigma_u)}\frac{h(u/i) u}{\cos(\pi u)} 
\\ \times \frac{1}{2\pi i} \int_{(\sigma_{v, 2})}  
\frac{\Gamma\left(u-v\right)\Gamma\left(1-s+v-\nu\right)\Gamma\left(1-s+v+\nu\right)}
{\Gamma\left(1+u+v\right)} 
n^v \sum_{m=1}^\infty \frac{C_\phi(m)\sigma_{1-2s}(n+m; \xi_*, M)}{m^{1-s+v}} \; dv\; du, 
\end{multline}
where $\Re(s) < \sigma_{v, 2}< \Re(u)$. 
\end{lemma}

\begin{proof}
By applying \eqref{e:Lq2} to \eqref{e:Phi_def}, and then by using \eqref{e:sum_cxiq}, 
and applying the integral representation of $F_3(s, u, \nu; x)$ in \eqref{e:F3}, 
we get \eqref{e:L2}.
\end{proof}

\section{Proof of Theorem~\ref{thm:upperbound_first}}
Recall $h= h_{T, \alpha}$, with $0<\alpha \le 1$.
\subsection{The holomorphic upper bounds}
When $\phi$ is a holomorphic cusp form of weight $k$, level $N$ with the trivial central character modulo $N$, 
recall that 
\begin{equation}
\phi(z) = \sum_{n=1}^\infty c_\phi(n)e^{2\pi inz} 
\end{equation}
and $C_\phi(n) = c_\phi(n)n^{-\frac{k-1}{2}}$. 
By Theorem~\ref{thm:first_holo}, 
\begin{equation}
K(1/2+it, \phi; n, h)
= M(1/2+it, \phi; n) + L_1^-(1/2+it, \phi; n) + L_1^+(1/2+it, \phi; n)
\end{equation}
where
\begin{multline}
M(1/2+it, \phi; n)
= \zeta^{(M)}(1+2it) \frac{C_\phi(n)}{n^{\frac{1}{2}+it}} 
H_0(h)
\\ + (2\pi)^{4it} \zeta(1-2it)M^{-2it} \prod_{p\mid M} (1-p^{-1})
\frac{C_\phi(n)}{n^{\frac{1}{2}-it}} 
H_{\frac{k-1}{2}}(1/2+it; h), 
\end{multline}
where $H_0(h)$ and $H_{\frac{k-1}{2}}(1/2+it; h)$ are given in \eqref{e:H0} and \eqref{e:Hnu_h} respectively,
\begin{multline}\label{e:L-_1/2}
L_1^-(1/2+it, \phi; n)
= -\frac{(-1)^{k} (2\pi)^{2it} \cos(\pi it)}{\pi} 
\frac{4}{2\pi i} \int_{(\epsilon)} \frac{h(u/i) u \tan(\pi u)}
{\Gamma\left(u+it+\frac{k}{2}\right) \Gamma\left(-u+it+\frac{k}{2}\right)}
\\ \times 
\frac{1}{2\pi i} \int_{(-\epsilon-\sigma_1)} \Gamma\left(u-v\right) \Gamma\left(-u-v\right) 
\Gamma\left(v-it+\frac{k}{2}\right) \Gamma\left(v+it+\frac{k}{2}\right)
\\ \times n^v \sum_{m=1}^{n-1} \frac{c_\phi(n-m) \sigma_{-2it}(m; M)}{m^{-it+v+\frac{k}{2}}} \; dv \; du, 
\end{multline}
for some sufficiently small $\sigma_1, \epsilon>0$, and 
\begin{multline}\label{e:L+_1/2}
L_1^+(1/2+it, \phi; n)
= (2\pi)^{2it} i^k 
\frac{4}{2\pi i} \int_{(\sigma_u)}\frac{h(u/i) u}{\cos(\pi u)}
\frac{1}{\Gamma\left(-u+it+\frac{k}{2}\right)\Gamma\left(u+it+\frac{k}{2}\right)}
\\ \times \frac{1}{2\pi i} \int_{(\sigma_v)} \frac{\Gamma\left(-it+\frac{k}{2}+v\right) \Gamma\left(it+\frac{k}{2}+v\right)\Gamma\left(-v+u\right)}
{\Gamma\left(1+u+v\right)} 
n^{v} \sum_{m=1}^\infty \frac{c_\phi(m+n)\sigma_{-2it}(m; M)}{m^{-it+v+\frac{k}{2}}} 
\; dv\; du, 
\end{multline}
for $\frac{5}{4}+3\epsilon < \sigma_u < \frac{3}{2}$ and $\frac{1}{2} < \sigma_v < \sigma_u$. 
Here, when $\frac{M}{\prod_{p\mid M} p}\mid n$, 
\begin{equation}
\sigma_{-2it}(n; M) = \sigma_{-2it}(n; 1_N, M)
=
\sigma_{-2it}(n; 1_M) P_M(2it, n)
\prod_{p\mid M} p^{-2it\ord_p(n)}
\end{equation}
and $\sigma_{-2it}(n; M)=0$ otherwise, 
%
as given in Lemma~\ref{lem:sum_cxiq}. 

Now we fix $T\gg 1$, $0< \alpha \leq 1$ and take $h(r)=h_{T, \alpha}(r)$ as given in \eqref{e:testh_f}. 
An upper bound for $L_1^-$ is given in the following lemma.

\begin{lemma}\label{lemma:5.1}
For any $\epsilon>0$,  we have
\begin{equation}
L_1^-(1/2+it, \phi; n) \ll M^\epsilon n^{\frac{1}{2}+\epsilon} T^{\alpha+\beta+\epsilon}, 
\end{equation}
where $|t|=T^\beta$ for $\beta <1$. 
When $t=0$ we take $\beta=-\infty$.
\end{lemma}

\begin{proof}
Recalling \eqref{e:L-_1/2}, we move the $v$ line of integration to $\Re(v)=1/2-k/2-\epsilon$.
On this line of integration, 
\begin{equation}
\left|n^v \sum_{\ell=1}^{n-1} \frac{c_\phi(n-\ell)\sigma_{-2it}(\ell; M)}{\ell^{v-it+\frac{k}{2}}}\right|
\ll M^\epsilon n^{\frac{1}{2}-\frac{k}{2}-\epsilon} \sum_{\ell=1}^{n-1} \frac{(n-\ell)^{\frac{k-1}{2}}}{\ell^{\frac{1}{2}-\epsilon}} 
\ll M^\epsilon n^{\frac{1}{2}+\epsilon}. 
\end{equation}

We now consider the ratio of gamma functions in the integrand of $L_1^-(1/2+it, \phi; n)$ in \eqref{e:L-_1/2}: 
\begin{equation}\label{e:integrand_L-1/2}
\frac{u\tan(\pi u) \Gamma\left(-v+u\right) \Gamma\left(-v-u\right) 
\Gamma\left(v-it+\frac{k}{2}\right) \Gamma\left(v+it+\frac{k}{2}\right)}
{\Gamma\left(\frac{1}{2}+it\right)\Gamma\left(\frac{1}{2}-it\right) \Gamma\left(u+it+\frac{k}{2}\right) \Gamma\left(-u+it+\frac{k}{2}\right)}.
\end{equation}
Let $v=\sigma_v+ir$ and $u=\sigma_u+i\gamma$. 
By Stirling's formula, the exponential part of the gamma factor contribution is given by 
\begin{equation}
\exp\left(-\frac{\pi}{2}\big(-2|t|-2\max\{|\gamma|, |t|\} + 2\max\{|r|,|\gamma|\}+2\max\{|r|, |t|\}\big)\right).
\end{equation}
So it is exponentially decreasing unless $|r| \leq |t| \leq |\gamma|$.
In this case, the polynomial contribution of \eqref{e:integrand_L-1/2} is given by 
\begin{equation}\label{e:holomorphic_upper_poly1}
\frac{|\gamma| |\gamma-r|^{\sigma_u-\sigma_v-\frac{1}{2}} |\gamma +r|^{-\sigma_u-\sigma_v-\frac{1}{2}} 
|r-t|^{\sigma_v+\frac{k}{2}-\frac{1}{2}} |r+t|^{\sigma_v+\frac{k}{2}-\frac{1}{2}}}
{|\gamma+t|^{\sigma_u+\frac{k}{2}-\frac{1}{2}}|-\gamma+t|^{-\sigma_u+\frac{k}{2}-\frac{1}{2}}}.
\end{equation}
Recall we have $|t|= T^\beta$, with $\beta < 1$.   
Recall also, that because of our choice of the test function $h=h_{T, \alpha}$,
the integral has exponential decay unless $|\gamma-T| \ll T^\alpha$, with $0 <\alpha \le1$.  It follows that
if $\beta < 1-\epsilon$, for $\epsilon >0$ then as 
$|r| \leq |t| \leq |\gamma|$, for large $T$, 
the expression above in \eqref{e:holomorphic_upper_poly1} is bounded by
\begin{equation}
\ll 
T^{1-2\sigma_v-k}|r-t|^{\sigma_v+\frac{k}{2}-\frac{1}{2}} |r+t|^{\sigma_v+\frac{k}{2}-\frac{1}{2}}. 
\end{equation}
As this integrand is integrated over a $\gamma$ interval of length $\ll T^\alpha$ and over an $r$ interval of length $\ll T^\beta$, and $T^\beta$ dominates $|r|$ over most of this interval, the integrand before integrating is
\begin{equation}
\ll T^{1-2\sigma_v-k+\beta(2\sigma_v+k-1)} = T^{2\epsilon(1-\beta)},
\end{equation}
as $\sigma_v=1/2-k/2-\epsilon$, and after integrating, the integral is 
\begin{equation}
\ll T^{\alpha+\beta+\epsilon},
\end{equation}
with a different $\epsilon >0$.
\end{proof}

\begin{lemma}\label{lemma:5.2}
For any $\epsilon>0$, we have 
\begin{equation}\label{e:L1+_1/2+it_bound}
L_1^+(1/2+it, \phi; n) \ll 2^{\frac{k}{2}} M^{\epsilon} n^{\frac{1}{2}+\epsilon} T^{1+\epsilon}
\end{equation}
if $|t| =T^\beta$, with $ \beta <1$ and the choice of the test function $h$ as given in \eqref{e:testh_f}.
\end{lemma}

\begin{proof}
Recalling \eqref{e:L+_1/2}, 
%
we move the $u$ line of integration to $1+\epsilon < \Re(u) = 1-\sigma_v+\sigma_u+\epsilon < \frac{3}{2}$ 
for some sufficiently small $\epsilon>0$. 
Since $\sigma_v< \sigma_u$, $1-\sigma_v+\sigma_u+\epsilon > 1+\epsilon$. 
By changing the variable $v+\frac{k}{2}$ to $v'$, so $\sigma_{v'}=k/2+1+\epsilon$, we get, after dropping the $'$, 
\begin{multline}\label{e:L1+_holo_upper1}
L_1^+(1/2+it, \phi; n)
= (2\pi)^{2it} i^k 
\frac{4}{2\pi i} \int_{(1+\epsilon_1)}\frac{h(u/i) u}{\cos(\pi u)}
\frac{1}{\Gamma\left(-u+it+\frac{k}{2}\right)\Gamma\left(u+it+\frac{k}{2}\right)}
\\ \times \frac{1}{2\pi i} \int_{(k/2+1+\epsilon)} \frac{\Gamma\left(-it+v\right) \Gamma\left(it+v\right)\Gamma\left(-v+\frac{k}{2}+u\right)}
{\Gamma\left(1-\frac{k}{2}+u+v\right)} 
\sum_{m=1}^\infty \frac{c_\phi(m+n)\sigma_{-2it}(m; M)}{m^{-it+v}n^{-v+\frac{k}{2}}} 
\; dv\; du. 
\end{multline}

We now estimate $L_1^+(1/2+it, \phi; n)$ from above.
We separate the series into two pieces: 
\begin{equation}
\sum_{m=1}^\infty \frac{\sigma_{-2it}(m; M)}{m^{v-it}}\frac{c_\phi(n+m)}{n^{-v+\frac{k}{2}}}
= \sum_{m=1}^{n-1}  \frac{\sigma_{-2it}(m; M)}{m^{v-it}}\frac{c_\phi(n+m)}{n^{-v+\frac{k}{2}}}
+ \sum_{m=n}^\infty  \frac{\sigma_{-2it}(m; M)}{m^{v-it}}\frac{c_\phi(n+m)}{n^{-v+\frac{k}{2}}}. 
\end{equation}
For the first piece, we move the $v$ line of integration to $\sigma_v=\frac{1}{2}+\epsilon$ and get
\begin{multline}\label{e:OD+_first_short}
\left|\sum_{m=1}^{n-1}  \frac{\sigma_{-2it}(m; M)}{m^{v-it}}\frac{c_\phi(n+m)}{n^{-v+\frac{k}{2}}}\right|
\ll n^{1/2 -k/2+\epsilon'}M^\epsilon \sum_{m=1}^{n-1}  \frac{(n+m)^{\frac{k-1}{2}+\epsilon}}{m^{1/2+\epsilon}}
\\ \ll (nM)^{\epsilon} 2^{\frac{k}{2}} \sum_{m=1}^{n-1}\frac{1}{m^{\frac{1}{2}+\epsilon}}
\ll n^{\frac{1}{2}+\epsilon}2^{\frac{k}{2}}M^\epsilon. 
\end{multline}
For the second piece, we keep the $v$ line of integration as $\sigma_v = 1+\frac{k}{2}+\epsilon$ and get
\begin{multline}\label{e:OD+_first_long}
\left|\sum_{m=n}^\infty  \frac{\sigma_{-2it}(m; M)}{m^{v-it}}\frac{c_\phi(n+m)}{n^{-v+\frac{k}{2}}}\right|
\ll n^{1+\epsilon}M^\epsilon \sum_{m=n}^\infty  \frac{(n+m)^{\frac{k-1}{2}+\epsilon}}{m^{1+\frac{k}{2}+\epsilon}}
\\ \ll n^{1+\epsilon}M^\epsilon 2^{\frac{k}{2}} \sum_{m=n}^\infty \frac{1}{m^{\frac{3}{2}+\epsilon}}
\ll M^\epsilon 2^{\frac{k}{2}} n^{\frac{1}{2}+\epsilon}. 
\end{multline}

Now we study the contribution from the gamma factors. 
By the definition of $h=h_{T, \alpha}$, if we write $u = \sigma_u + i \gamma$, 
there is quadratic exponential decay in $|\gamma|$, 
when $|T-|\gamma| |>T^\alpha$.  

Assume that $|T-|\gamma|| \leq T^{\alpha}$. 
Write $v = \sigma_v + ir$.  
By Stirling's formula, the exponential part of the gamma factor 
 in the integrand in \eqref{e:L1+_holo_upper1} and the $\cos(\pi u)$ contribution is given by 
\begin{equation}
\exp\left(-\frac{\pi}{2}\left( 2|\gamma|  - 2 \max(|\gamma|,|t|) + 2 \max(|t|,|r|) +|\gamma-r| -|\gamma+r|\right)\right).
\end{equation}
Recall $|t| = T^\beta$ for $\beta<1$.
If $|t|>|\gamma|$, since $|T-|\gamma|| \leq T^\alpha$ 
\begin{equation}
T^\beta = |t| > T(1-cT^{\alpha-1})
\end{equation}
so 
\begin{equation}
\beta > 1+\log(1-cT^{\alpha-1}) > 1-cT^{\alpha-1}.
\end{equation}
This cannot happen for arbitrarily large $T$ since $\beta$ is independent of $T$. 
Therefore we are reduced to the case that 
\begin{equation}
|t| \leq |\gamma|. 
\end{equation}

The worst case is when $r$ and $\gamma$ have the same sign,  which gives us
\begin{equation}
\exp\left(-\frac{\pi}{2}\big(2\max\{|t|, |r|\}+ \max\{|r|, |\gamma|\}-\min\{|r|, |\gamma|\}-|r|-|\gamma|\big)\right)
\end{equation}
If $|\gamma|< |r|$, then as $|\gamma|\geq |t|$,  we get
\begin{equation}
\exp\left(-\frac{\pi}{2}\big(2|r|-2|\gamma|\big)\right),
\end{equation}
which has exponential decay. 

Now we are in the situation when $|\gamma|\geq |r|$. 
Then the exponential term is
\begin{equation}
\exp\left(-\frac{\pi}{2}\big(2\max\{|t|, |r|\}-2|r|\big)\right),
\end{equation}
which has exponential decay unless $|t|\leq |r|$.
Thus the only case there is no exponential decay is $|t|\leq |r|\leq |\gamma|$ and $r$ and $\gamma$ have the same sign. 
In this case
\begin{multline}
\frac{u \Gamma\left(v-it\right) \Gamma\left(v+it\right) \Gamma\left(-v+u+\frac{k}{2}\right)}
{\cos(\pi u)\Gamma\left(-u+it+\frac{k}{2}\right)\Gamma\left(u+it+\frac{k}{2}\right)\Gamma\left(v+u+1-\frac{k}{2}\right)}
\\ \ll |\gamma| |-\gamma+t|^{\sigma_u-\frac{k}{2}+\frac{1}{2}} |\gamma+t|^{-\sigma_u-\frac{k}{2}+\frac{1}{2}}
|r-t|^{\sigma_v-\frac{1}{2}} |r+t|^{\sigma_v-\frac{1}{2}}
|\gamma-r|^{\frac{k}{2}+\sigma_u-\sigma_v-\frac{1}{2}}
|\gamma+r|^{-\sigma_v-\sigma_u+\frac{k}{2}-\frac{1}{2}}.
\end{multline}
Consider the $\sigma_u$-power pieces: 
\begin{equation}
\left(\frac{|\gamma-r|}{|\gamma+r|}\right)^{\sigma_u} \left(\frac{|\gamma-t|}{|\gamma+t|}\right)^{\sigma_u}
= 
\left|\frac{1-\left|\frac{r}{\gamma}\right|}{1+\left|\frac{r}{\gamma}\right|}\right|^{|\gamma|\frac{\sigma_u}{|\gamma|}} 
\left|\frac{1-\frac{t}{\gamma}}{1+\frac{t}{\gamma}}\right|^{|\gamma|\frac{\sigma_u}{|\gamma|}}
\ll e^{-(2|r|-2|t|) \frac{\sigma_u}{|\gamma|}},
\end{equation}
because $\gamma$ and $r$ have the same sign.   We have also chosen the case of weakest decay, which is when $t,\gamma$ have opposite signs.

Recall now that $\sigma_v = 1+\frac{k}{2}+\epsilon$ or $\sigma_v = \frac{1}{2}+\epsilon$ and $\sigma_u>1+\epsilon$.   
It follows that $\Re \,(-v+u+\frac{k}{2}) >0$ and thus we may move the $u$ line in a positive direction to $\sigma_u = T^\alpha$.  
(This is as far as we can move $u$ and still maintain the quadratic exponential decay of $h$).   
After doing so we obtain
\begin{equation}
e^{-(2|r|-2|t|) \frac{\sigma_u}{|\gamma|}}\ll e^{-\frac{(2|r|-2|t|)}{T^{1-\alpha}}},
\end{equation}
which is exponentially decaying unless $||r|-|t|| \ll T^{1-\alpha+\epsilon}$. 
It follows that $L_1(1/2+it, \phi; n)$ can be bounded above by
\begin{multline}
\int_{|T-|\gamma|| < T^{\alpha}} 
\int_{\substack{T+cT^{\alpha} \geq |r| \geq T^\beta=|t|, \\ ||r|-T^\beta| \ll T^{1-\alpha+\epsilon}}}
\left|\frac{u \Gamma\left(v-it\right) \Gamma\left(v+it\right) \Gamma\left(-v+u+\frac{k}{2}\right)}
{\cos(\pi u)\Gamma\left(-u+it+\frac{k}{2}\right)\Gamma\left(u+it+\frac{k}{2}\right)\Gamma\left(v+u+1-\frac{k}{2}\right)}\right|
\; dr\; d\gamma 
\\ \ll T^{1-k+1+k -2\sigma_v-1 + \beta(\sigma_v-\frac{1}{2}) + (1-\alpha)(\sigma_v+\frac{1}{2})+\alpha+\epsilon}
= T^{1+(\beta-1-\alpha)(\sigma_v-\frac{1}{2})+\epsilon}.
\end{multline}
A $\beta <1$, the largest exponent occurs in the $\sigma_v = \frac{1}{2} +\epsilon$ case, 
leaving us with an upper bound of $T^{1+\epsilon}$.

Finally, combining with \eqref{e:OD+_first_short} and \eqref{e:OD+_first_long}, we get \eqref{e:L1+_1/2+it_bound}. 
This completes the proof of the lemma. 
\end{proof}

\subsection{The nonholomorphic upper bounds}
In this section we will find upper bounds for $L_{1}^+(s, \phi; n)$, defined in \eqref{e:L1+nonholo_holo}, 
$L_1^-(s, \phi, \nu; n)$ defined in \eqref{e:L1-nonholo}, and $L_2 (s, \phi; n)$, defined in  \eqref{e:L2}, 
at $s=\frac{1}{2}+it$.  
Virtually identical arguments as in the holomorphic case, with $(k-1)/2$ replaced by $\nu$, lead to the upper bounds
\begin{align}
& L_1^-(1/2+it, \phi; n)
\ll M^\epsilon n^\epsilon T^{\alpha +\beta +\epsilon}\sum_{m=1}^{n-1} \frac{|C_\phi(n-m)|}{m^{1/2+\epsilon}}; 
\label{e:L1-_nonoholo_init_upper}
\\ & L_1^+(1/2+it, \phi; n)\ll M^\epsilon n^\epsilon T^{1+\epsilon}\sum_{m=1}^{n} \frac{|C_\phi(n+m)|}{m^{1/2+\epsilon}}. 
\label{e:L1+_nonoholo_init_upper}
\end{align}
Here we've used the fact that 
\begin{equation}
\sigma_{-2it}(n-m; M)\ll M^\epsilon n^\epsilon \quad \text{and} \quad \sigma_{-2it}(m; M)\ll M^\epsilon m^\epsilon.
\end{equation}
Note that the factor $2^{\frac{k}{2}}$ in the error term is no longer there, as  $|2^{\nu}|=1$ for $\nu\in i\R$.
Also note that if $C_\phi (r) \ll r^\theta$, where $\theta$ is the best progress toward the Ramanujan conjecture, 
then the two expressions in  \eqref{e:L1-_nonoholo_init_upper} and \eqref{e:L1+_nonoholo_init_upper} satisfy the bounds
\begin{equation}\label{nonholup}
 L_1^-(1/2+it, \phi; n) \ll M^\epsilon n^{\theta +\frac12+\epsilon} T^{\alpha+\beta+\epsilon}
 \end{equation}
 and
 \begin{equation}\label{nonholup2}
 L_1^+(1/2+it, \phi; n) \ll M^\epsilon n^{\theta+\frac12+\epsilon}T^{1+\epsilon}.
 \end{equation}


Now we estimate $L_2 (1/2+it, \phi; n)$ and
will show that it has exponential decay in $T$ multiplied by $n^{\theta+ \frac12+\epsilon}M^\epsilon$.
Recalling \eqref{e:L2}, and writing $\nu=it'$, 
\begin{multline}\label{e:L2_upper1}
L_2 (1/2+it, \phi; n)
= (2\pi)^{2it} \frac{c_\phi(-1)}{c_\phi(1)}\frac{\cos(\pi it')}{\pi} 
\frac{4}{2\pi i} \int_{(\sigma_u)}\frac{h(u/i) u}{\cos(\pi u)} 
\\ \times \frac{1}{2\pi i} \int_{(\sigma_{v, 2})}  
\frac{\Gamma\left(u-v\right)\Gamma\left(\frac{1}{2}-it-it'+v\right)\Gamma\left(\frac{1}{2}-it+it'+v\right)}
{\Gamma\left(1+u+v\right)} 
n^v \sum_{m=1}^\infty \frac{C_\phi(m)\sigma_{-2it}(n+m; M)}{m^{\frac{1}{2}-it+v}} \; dv\; du, 
\end{multline}
where $\frac{1}{2} < \sigma_{v, 2}< \sigma_u<\frac{3}{2}$. 
Write $u=\sigma_u+i\gamma$ and $v=\sigma_v+ir$. 
We first note that, by Stirling's formula, the exponential part of the gamma factor in the integrand in \eqref{e:L2_upper1} contribution is given by 
\begin{equation}
\exp\left(-\frac{\pi}{2}\left(-2|t'|+2|\gamma|+|\gamma-r| +2\max\{|r-t|, |t'|\}-|\gamma+r|\right)\right).
\end{equation}
Also, recall we have assumed that $|t| = T^{\beta}$ with $\beta<1$ and $|t'|\ll T^{1-\epsilon}$, 
and so is small compared to $\gamma$, which is on the order of $T$.

Clearly the decay is minimized when $\gamma$ and $r$ are the same sign. 
If $|r| >|\gamma|$ then the quantity multiplied by $-\pi/2$ in the exponent becomes 
\begin{equation}
-2|t'|+2|\gamma|+|r| -|\gamma| +2|r| -2|t| -|\gamma|-|r| = 2|r| -2|t| -2|t'|.
\end{equation}
There is thus exponential decay in $T$ as $|r| >T$ and $|t|,|t'| < T^{1-\epsilon}$.

If $|r| \le |\gamma|$ then the quantity multiplied by $-\pi/2$ in the exponent becomes 
\begin{equation}
-2|t'|+2|\gamma|+|\gamma|- |r| +2\max\{|r-t|, |t'|\} -|\gamma|-|r|=-2|t'|+2|\gamma|- 2|r| +2\max\{|r-t|, |t'|\}.
\end{equation}
The worst case (minimal decay) occurs when $r, t$ are the same sign.   
If $|r| \ge |t|$ and $|r-t| \ge |t'|$ the above becomes
\begin{equation}
-2|t'|+2|\gamma|- 2|r| +2|r| -2|t| =2|\gamma|-2|t'|-2|t|.
\end{equation}
This has exponential decay in $T$ as $|t|,|t'| < T^{1-\epsilon}$ and $\gamma \gg T$.
 
If $|r| \ge |t|$ and $|r-t| <  |t'|$, the above becomes
\begin{equation}
-2|t'|+2|\gamma|- 2|r| +2 |t'| =2|\gamma|- 2|r|.
\end{equation}
This must also have exponential decay in T as $|r-t| <  |t'|$, with $r, t$ the same sign and $|t|, |t'| < T^{1-\epsilon}$,
while $\gamma \gg T$.

We are now reduced to the last case, where $r, t$ have the same sign and $|r| <|t|$.  
The quantity multiplied by $-\pi/2$ in the exponent is
\begin{equation}
-2|t'|+2|\gamma|- 2|r| +2\max\{|r-t|, |t'|\}. 
\end{equation}
As all terms except $|\gamma|$ have absolute value less than $T^{1-\epsilon}$, we once again have exponential decay in $T$. This exponentially decaying quantity is multiplied by 
\begin{equation}\label{nonholup3}
n^{v}\sum_{m=1}^\infty \frac{|C_\phi(m)\sigma_{-2it}(n+m;M)|}{m^{\frac{1}{2}+v}} \ll n^{\frac{1}{2}+\theta+\epsilon}M^\epsilon,
\end{equation}
where we have set  $\sigma_v =\frac{1}{2} +\epsilon$.  

The above, together with Lemmas~\ref{lemma:5.1} and \ref{lemma:5.2} completes the proof of  statement \eqref{e:first_est} in Theorem~\ref{thm:upperbound_first}, writing the first moment as a main term plus potential error terms.

To determine the requirements for $T$ under which the errors are of lower order of magnitude than the main term, we refer first to \eqref{e:main_asymp} and  \eqref{e:main_asymp2}, to see that when $t=0$ and $t\ne 0$,
$M(1/2 +it, \phi; n)$, given by \eqref{e:M_holo} when $\phi$ is holomorphic, and   \eqref{e:Mdef_nonholo} when $\phi$ is non-holomorphic,satisfies, for $0 < \epsilon < \alpha$,
\begin{equation}
M(1/2 +it, \phi; n)\gg_{\alpha, \epsilon}
\frac{C_\phi(n)}{n^{\frac{1}{2}} }T^{1+\alpha - \epsilon}.
\end{equation}
Referring to Lemma~\ref{lemma:5.1},Lemma~\ref{lemma:5.2} and \eqref{nonholup},\eqref{nonholup2}, \eqref{nonholup3}
we see that the error is 
\begin{equation}
O_{\alpha, \epsilon_1,\epsilon_2,\epsilon_3}\left( M^{\epsilon_1} 2^{\frac{k}{2}} n^{\frac12 + \theta+ \epsilon_2}T^{\max(1,\alpha+\beta)+\epsilon_3} \right).
\end{equation} 
Combining the previous two lines we see that the main term dominates the error when 
\begin{equation}
C_\phi(n)T \gg \left( M^\epsilon 2^{\frac{k}{2}}n^{1+\theta + \epsilon}\right)^{\frac{1}{\min(\alpha, 1-\beta)}+\epsilon},
\end{equation}
for sufficiently small $0< \epsilon < \min(\alpha, 1-\beta)$.  This establishes the dominance of the main term over the error given in line \eqref{bound} of Theorem~\ref{thm:upperbound_first}.
\section{Proof of Corollary~\ref{cor:first_eis}}\label{ss:proof_cor_first_eis}
For this corollary, we take $\phi$ as the completed Eisenstein series for $\SL_2(\Z)$: 
\begin{multline}\label{e:Eis_1}
\phi(z) = \frac{1}{2} E^*(z, 1/2+it)
\\ =\frac{1}{2} \Gamma\left(\tfrac{1}{2}+it\right) \zeta(1+2it) \pi^{-\frac{1}{2}-it} y^{\frac{1}{2}+it} 
+ \frac{1}{2} \Gamma\left(\tfrac{1}{2}-it\right)\pi^{-\frac{1}{2}+it}\zeta(1-2it) y^{\frac{1}{2}-it}
\\ + \sum_{n\neq 0} \sigma_{-2it} (n)|n|^{it} \sqrt{y} K_{it}(2\pi|n|y) e^{2\pi inx}.
\end{multline}
Since each Maass cusp form $u_j$ for $\SL_2(\Z)$ is taken to be an eigenfunction for Hecke operators, we get 
\begin{equation}
\scrL(s, \phi\times u_j) = \zeta(2s)\rho_j(1) \sum_{m=1}^\infty \frac{\sigma_{-2it}(m)m^{it} \lambda_j(m)}{m^s}
= \rho_j(1) L(s+it, u_j) L(s-it, u_j)
\end{equation}
and 
\begin{equation}\label{e:scrL_infty_level1}
\scrL(s, ir;\phi) 
= \frac{2\zeta(s+it+ir)\zeta(s+it-ir)\zeta(s-it+ir)\zeta(s-it-ir)}{\Gamma\left(\frac{1}{2}-ir\right)\zeta(1-2ir) \pi^{-\frac{1}{2}+ir}} 
\end{equation}

In the first moment $K(s, \phi; n, h)$ \eqref{e:K_firstmoment_def} for $\SL_2(\Z)$, we consider the continuous spectrum: 
for $\Re(s)>1$, by \eqref{e:scrL_infty_level1}
\begin{multline}
K_{\cont}(s, \phi; n, h) 
\\ = \frac{1}{4\pi} \int_{-\infty}^\infty \frac{h(r)}{\cosh(\pi r)} 
\frac{4\zeta(s+it+ir)\zeta(s+it-ir)\zeta(s-it+ir)\zeta(s-it-ir)}{\zeta^*(1+2ir)\zeta^*(1-2ir)}
\sigma_{-2ir}(n)n^{ir} \; dr. 
\end{multline}
Here $\zeta^*(s) = \Gamma\left(\frac{s}{2}\right)\zeta(s)\pi^{-\frac{s}{2}}$. 
As we move $s$ to $\Re(s)=1/2$, we pass over the poles of the product of four zeta functions in the numerator
at $s\pm it\pm ir=1$. 
Change the variable $ir=z$. 
When $\Re(s)=1+\epsilon$ for a sufficiently small $\epsilon>0$,
we bend the $z$-line of integration to the right, over $s$, passing over the poles of $\zeta(s\pm it-z)$ 
and collecting residues with negative signs.  
We then continue $s$ back to $\Re(s)=1/2$ and  bend the $z$ line of integration back to $\Re(z) =0$, 
collecting more residues, which has the effect of doubling the original residue.   
The result is
\begin{multline}
K_{\cont}(s, \phi; n, h) 
\\ =\frac{1}{4\pi i} \int_{\Re(z)=0} \frac{h(z/i)}{\cos(\pi z)} 
\frac{4\zeta(s+it+z)\zeta(s+it-z)\zeta(s-it+z)\zeta(s-it-z)}{\zeta^*(1+2z)\zeta^*(1-2z)}
\sigma_{-2z}(n)n^{z} \; dz
\\ + 4\bigg\{\frac{h((1-s+ it)/i)}{\cos(\pi(1-s+ it))} 
\frac{\zeta(1+2it)\zeta(-1+2s)}{\zeta^*(3-2s+2it)\Gamma\left(-\frac{1}{2}+s-it\right)\pi^{\frac{1}{2}-s+it}}
\\ + \frac{h((1-s-it)/i)}{\cos(\pi(1-s- it))} 
\frac{\zeta(1-2it)\zeta(-1+2s)}{\zeta^*(3-2s-2it)\Gamma\left(-\frac{1}{2}+s+it\right)\pi^{\frac{1}{2}-s-it}}
\bigg\}.
\end{multline}
Taking $s=1/2$,  using the fact that $\zeta(0)=-\frac{1}{2}$, and applying the reflection formula for gamma functions, 
\begin{multline}
K_{\cont}(1/2, \phi; n, h) 
\\ =\frac{1}{4\pi i} \int_{\Re(z)=0} \frac{h(z/i)}{\cos(\pi z)} 
\frac{4\zeta(1/2+it+z)\zeta(1/2+it-z)\zeta(1/2-it+z)\zeta(1/2-it-z)}{\zeta^*(1+2z)\zeta^*(1-2z)}
\sigma_{-2z}(n)n^{z} \; dz
\\ - 2\bigg\{h((1/2+ it)/i)\frac{\zeta(1+2it)}{\zeta(2+2it)}
+h((1/2-it)/i)\frac{\zeta(1-2it)}{\zeta(2-2it)}\bigg\}.
\end{multline}

Taking $s=1/2$ for the main term $M(s, \phi; n)$, by \eqref{e:M_nonholo_t=0}, 
we get \eqref{e:M_Eis_s=1/2}. 
Recalling \eqref{e:M+} and taking $s=1/2$, we get \eqref{e:Mpm_s=1/2}. 
Combining them together, we get \eqref{e:first_eis_level1}. 

\section{Proof of Corollary~\ref{cor:determine}}\label{ss:proof_cor_detrmine}
The left hand side of Theorem~\ref{thm:upperbound_first} consists of two pieces, 
corresponding to the discrete and the continuous spectrum.   
We will begin by finding an upper bound for the continuous part of the first moment when the level is $M$ .
In both the holomorphic and Maass form cases, this is given by
\begin{equation}
K_{\cont} (n,\phi)
=\sum_{\cuspa} \frac{1}{4\pi} \int_{-\infty}^\infty \frac{h_{T, \alpha} (r)}{\cosh(\pi r)}
\tau_{\cuspa} (1/2-ir, n) \scrL(1/2, \phi\times E_{\cuspa}(*, 1/2+ir)) \; dr
\end{equation}
This is bounded from above in the following. 

\begin{lemma}\label{lem:cont}
Let $\phi$ be a newform of even weight $k$ if holomorphic and type $\nu$ if non-holomorphic, 
of level $N$, which divides $M$.  
Let $\kappa=k$ if $\phi$ is holomorphic, and $\kappa = |\nu|$ if $\phi$ is non-holomorphic.
For any  $\epsilon >0$, and for $T \gg \kappa$,
\begin{equation}
K_{\cont}(n,\phi) \ll N^{\frac{1}{2}}\kappa  T (Tn \kappa M)^{\epsilon}
\end{equation}
\end{lemma}

\begin{proof}
We have
\begin{equation}
\left|\scrL(1/2, \phi\times E_{\cuspa}(*, 1/2+ir))\right| 
\ll \left|\tau_{\cuspa}(1/2+ir,1)\right| \left| L(1/2+ir, \phi)\right| \left|L(1/2-ir, \phi)\right|. 
\end{equation}
In \cite[Lemma~3.4]{B}, it is shown (after converting to our notation) that
\begin{equation}
\frac{\left|\zeta(1+ir)\right|^2}{\cosh(\pi r)}
\sum_\cuspa \left| \tau_\cuspa \left(1/2+ir, n\right)\right|^2
\ll_\epsilon ((1+|r| )|n|)^{\epsilon}
\end{equation}
for any $\epsilon>0$, with the implied constant depending only on $\epsilon$.

Applying Cauchy-Schwartz, and the lower bound  $|\zeta(1+2ir)| \gg (\log(|r|)^{-1}$, for large $|r|$, 
to the sum over cusps $\cuspa$ gives us
\begin{multline}
\left|\frac{1}{\cosh(\pi r)} \sum_\cuspa  \tau_\cuspa \left(1/2-ir, n\right)\tau_{\cuspa}\left(1/2+ir,1\right) \right| 
\\ \leq 
\left(\sum_\cuspa \frac{ |\tau_\cuspa \left(1/2-ir, n)\right|^2}{\cosh(\pi r)} \right)^{1/2}  
\left( \sum_\cuspa \frac{|\tau_\cuspa \left(1/2+ir, 1\right|^2}{\cosh(\pi r)}\right)^{1/2}  \ll_\epsilon ((1+|r| )|n|M)^{\epsilon}.
\end{multline}

Combining the above,
\begin{multline}\label{contup}
\frac{1}{\cosh (\pi r)}\sum_{\cuspa} 
\frac{1}{4\pi} \int_{-\infty}^\infty 
\frac{h_{T, \alpha}(r)}{\cosh(\pi r)}
\tau_{\cuspa} (1/2-ir, n) 
\scrL(1/2, \phi\times E_{\cuspa}(*, 1/2+ir)) \; dr \\ \ll  ((1+|r| )|n|M)^{\epsilon}
|\int_{-T^\alpha}^{T^\alpha} \left|  L(1/2+ir, \phi) L(1/2-ir, \phi) \right|\,dr.
\end{multline}

Recall the definition of $\kappa$.
We have the well known  weak Lindel\"of on average result that for $T \gg (N^{\frac{1}{2}} \kappa)^{1+\epsilon}$,   
\begin{equation}
\int_{-T}^T L(1/2+ir, \phi) L(1/2-ir, \phi)\; dr 
\ll (N^{\frac{1}{2}}\kappa T)^{1+\epsilon}.
\end{equation}
This follows because the square root of the analytic conductor of  $L(1/2\pm it, \phi)$ is 
$\ll (N^{\frac{1}{2}}\kappa T)^{1+\epsilon}$.
The approximate functional equation can be then used to write each $L(1/2\pm it, \phi)$ as a Dirichlet polynomial 
of length $(N^{\frac{1}{2}}\kappa T)^{1+\epsilon}$.  
Then applying \cite[Theorem~9.1]{ik04}, it follows that 
\begin{equation}
K_{\cont} (n,\phi) 
\ll N^{\frac{1}{2} }\kappa T (Tn \kappa M)^{\epsilon}
\sum_{1\le m \ll (N^{\frac{1}{2}}\kappa T)^{1+\epsilon}}\frac{C(m)}{m}.
\end{equation}
A necessary additional ingredient is that for any $X \gg (N^{\frac{1}{2}} \kappa T)^{1+\epsilon}$, 
in both the Maass form and holomorphic cases, 
\begin{equation}
\sum_{1\le n \ll X}\frac{C(n)^2}{n} \ll X^{\epsilon}.
\end{equation}
This is well known, but also follows from \eqref{e:sum_Cphi_upper}.
The lemma then follows immediately. 
\end{proof}

To prove Corollary~\ref{cor:determine}, let $\phi_1$ and $\phi_2$ be two automorphic cusp forms 
(either holomorphic or non-holomorphic) 
and take the differences of the two spectral first moments \eqref{e:K_firstmoment_def}, 
with $s=\frac{1}{2}$ and $h=h_{T, \alpha}$. 
By \eqref{e:first_est} in Theorem~\ref{thm:upperbound_first}, 
\begin{multline}\label{e:phiphi'_first_difference}
K(1/2, \phi_1; n, h_{T, \alpha}) - K(1/2, \phi_2; n, h_{T, \alpha})
= \big(M(1/2, \phi_1; n)- M(1/2, \phi_2; n)\big) 
\\+ O_\epsilon(M^{\epsilon} 2^{\frac{\kappa}{2}} n^{\frac{1}{2}+\theta+\epsilon} T^{\max\{1, \alpha+\beta\}+\epsilon}). 
\end{multline}
For sufficiently large $T$, assume that 
\begin{equation}
\scrL(1/2, \phi_1 \times u_j) = \scrL(1/2, \phi_2 \times u_j),
\end{equation}
for $|r_j| \ll T^{1+\epsilon}$.   
Then the discrete contribution in the difference \eqref{e:phiphi'_first_difference} will vanish, 
up to an exponentially decreasing error term.   
The continuous part, as proved in Lemma~\ref{lem:cont},
contributes at most on the order of $M^{\frac{1}{2}}\kappa  T (Tn \kappa M)^{\epsilon}$.
 
Applying this estimate we have, after applying Theorem~\ref{thm:upperbound_first}, 
taking $\alpha =1$ and $\beta=-\infty$ (for $t=0$),
\begin{multline}\label{e:diff_main_upper}
\left|M(1/2, \phi_1; n)-M(1/2, \phi_2; n)\right| \ll 
M^{\frac{1}{2}} \kappa T (Tn\kappa M)^{\epsilon} 
+ M^{\epsilon} n^{\epsilon} T^{1+\epsilon} \big(E(\phi_1; n) + E(\phi_2; n)\big)
\\ + M^{\epsilon} T^{-A}  n^{1+\epsilon} 
\bigg(\left|\sum_{m>n} \frac{C_{\phi_1}(n+m)}{m^{\frac{3}{2}+\epsilon}}\right|
+\left|\sum_{m>n} \frac{C_{\phi_2}(n+m)}{m^{\frac{3}{2}+\epsilon}}\right|\bigg)
\end{multline}
for arbitrarily large $A$, 
where
\begin{equation}
E(\phi; n) = \left|\sum_{m=1}^{n} \frac{C_\phi(n-m)}{m^{\frac{1}{2}+\epsilon}} \right|
+ 2^{\frac{k}{2}} \left|\sum_{m=1}^n \frac{C_\phi(n+m)}{m^{\frac{1}{2}+\epsilon}}\right|.
\end{equation}
Note that when $n\ll T$, 
\begin{equation}
M^{\epsilon} T^{-A} 2^{\frac{k}{2}} n^{1+\epsilon} 
\bigg(\left|\sum_{m>n} \frac{C_{\phi_1}(n+m)}{m^{\frac{3}{2}+\epsilon}}\right|
+\left|\sum_{m>n} \frac{C_{\phi_2}(n+m)}{m^{\frac{3}{2}+\epsilon}}\right|\bigg)
\ll T^{-A} 2^{\frac{k}{2}} n^{\frac{1}{2}+\epsilon'} \ll T^{-A'}
\end{equation}
for arbitrarily large $A$ and $A'$. 
Also, applying \eqref{e:main_asymp}, when $n \ll T$, 
we see that 
\begin{equation}\label{e:diff_main_asymp}
\left|M(1/2, \phi_1; n)-M(1/2, \phi_2; n)\right| 
\gg T^2\log(T) \frac{\left|C_{\phi_1}(n)-C_{\phi_2}(n)\right|}{\sqrt{n}}.
\end{equation}
Combining \eqref{e:diff_main_upper} and \eqref{e:diff_main_asymp}, 
\begin{equation}\label{e:diff_Cphi_upper}
\left|C_{\phi_1}(n)-C_{\phi_2}(n)\right|
\ll \frac{\sqrt{n}}{T^2\log T} \bigg\{M^{\frac{1}{2}}\kappa T (T\kappa nM)^{\epsilon} 
+ M^{\epsilon} n^{\epsilon} T^{1+\epsilon} 
\big(E(\phi_1; n) + E(\phi_2; n)\big)\bigg\}.
\end{equation}

Note that the conductor for $L(s, \phi\times\phi)$ is $N^2\kappa^2$ and 
\begin{equation}
\frac{1}{2\pi i} \int_{(2)} \frac{L(s, \phi\times \phi)}{s(s+1)(s+2)} X^s \; ds
= \frac{1}{2}\sum_{n<X} C_\phi(n)^2 \left(1-\frac{n}{X}\right)^2.
\end{equation}
By moving the contour of the integral to $\Re(s)<1$, passing over the pole of $L(s, \phi\times \phi)$ at $s=1$, 
we get 
\begin{equation}
\frac{1}{2\pi i} \int_{(2)} \frac{L(s, \phi\times \phi)}{s(s+1)(s+2)} X^s \; ds= \frac{1}{2} L(1, \Sym^2\phi) X + O\left((XN\kappa)^{\frac{1}{2}+\epsilon}\right). 
\end{equation}
For $X\gg (XN\kappa)^{\frac{1}{2}+\epsilon}$, i.e., $X\gg (N\kappa)^{1+\epsilon}$, we get
\begin{equation}\label{e:sum_Cphi_upper}
\sum_{n\ll X} C_\phi(n)^2 \ll X^{1+\epsilon}. 
\end{equation}

We apply Cauchy-Schwartz to \eqref{e:diff_Cphi_upper}:
\begin{equation}
\left|C_{\phi_1}(n)-C_{\phi_2}(n)\right|^2
\ll n^{1+\epsilon} T^{-2+\epsilon} \bigg\{(\kappa^2M)^{1+\epsilon} 
+ 2M^{\epsilon} \big(E(\phi_1; n)^2 + E(\phi_2; n)^2\big)\bigg\}.
\end{equation}
Now we estimate 
\begin{equation}
E(\phi; n)^2 \leq 3 \bigg\{ \left|\sum_{m=1}^{n} \frac{C_\phi(n-m)}{m^{\frac{1}{2}+\epsilon}} \right|^2
+ 2^{k} \left|\sum_{m=1}^n \frac{C_\phi(n+m)}{m^{\frac{1}{2}+\epsilon}}\right|^2\bigg\}
\end{equation}

For the first term, by \eqref{e:sum_Cphi_upper} 
 for $n\gg (N\kappa)^{1+\epsilon}$, 
(necessary so the upper bound \eqref{e:sum_Cphi_upper} is valid)
\begin{equation}
\left|\sum_{m=1}^{n} \frac{C_\phi(n-m)}{m^{\frac{1}{2}+\epsilon}} \right|^2
\ll \sum_{m=1}^{n} C_\phi(m)^2 \ll n^{1+\epsilon}. 
\end{equation}
The second term, similarly for $n\gg (N\kappa)^{1+\epsilon}$, 
\begin{equation}
\left|\sum_{m=1}^n \frac{C_\phi(n+m)}{m^{\frac{1}{2}+\epsilon}}\right|^2
\ll \sum_{m=n+1}^{2n} C_\phi(m)^2 \ll n^{1+\epsilon}
\end{equation}
and we get 
\begin{equation}
E(\phi; n)^2\ll 2^k n^{1+\epsilon}. 
\end{equation}
Therefore, for $n\gg (\max\{N, N'\} \kappa)^{1+\epsilon}$, 
\begin{equation}\label{e:diff_Cphi_2_upper}
\left|C_{\phi_1}(n)-C_{\phi_2}(n)\right|^2
\ll n^{1+\epsilon} T^{-2+\epsilon} \bigg\{(\kappa^2M)^{1+\epsilon} 
+ 2^{k} M^{\epsilon} n^{1+\epsilon} \bigg\}.
\end{equation}

We now give the proof of a variation on \cite[Theorem~1]{Sen04}, adapted to level, 
and get a lower bound for a sum over $n$ of $|C_{\phi_1}(n)-C_{\phi_2}(n)|^2$.  
Choose $G(x)$ supported on $x\in [1/2, 1]$ and let 
\begin{equation}
g(s) = \int_0^\infty G(x)x^{s} \; \frac{dx}{x}
\end{equation}
be its Mellin transform. 
We further assume that $g(1)=1$. 
We have 
\begin{multline}
\frac{1}{2\pi i} \int_{(2)} \frac{L(s, \phi_1\times \phi_2)}{\zeta(2s)} g(s) X^s \; ds
= \sum_{n=1}^\infty C_{\phi_1}(n)C_{\phi_2}(n) 
\frac{1}{2\pi i} \int_{(2)}  g(s) \left(\frac{n}{X}\right)^{-s} \; ds
\\ = \sum_{n=1}^\infty C_{\phi_1}(n)C_{\phi_2}(n) G\left(\frac{n}{X}\right). 
\end{multline}
We move the contour of the integral to $\Re(s)=\frac{1}{2}+\epsilon$ and get
\begin{equation}
= \frac{X}{\zeta(2)} \Res_{s=1} L(s, \phi_1\times \phi_2) 
+ O(X^{\frac{1}{2}+\epsilon} (N_1N_2\kappa^2)^{\frac{1}{2}+\epsilon}). 
\end{equation}

When $\phi_1\neq \phi_2$, 
\begin{multline}
\sum_{n=1}^\infty |C_{\phi_1}(n)-C_{\phi_2}(n)|^2 G\left(\frac{n}{X}\right)
\\ = \sum_{n=1}^\infty |C_{\phi_1}(n)|^2 G\left(\frac{n}{X}\right) 
+ \sum |C_{\phi_2}(n)|^2 G\left(\frac{n}{X}\right)
- \sum_{n=1}^\infty C_{\phi_1}(n) \overline{C_{\phi_2}(n)}G\left(\frac{n}{X}\right) 
- \sum_{n=1}^\infty \overline{C_{\phi_1}(n)} C_{\phi_2}(n) G\left(\frac{n}{X}\right)
\\ = L(1, \Sym^2 \phi_1) X+ L(1, \Sym^2\phi_2) X 
+ O((X^{\frac{1}{2}+\epsilon} (N_1N_2\kappa^2)^{\frac{1}{2}+\epsilon}). 
\end{multline}
So when $X^{1-\epsilon} \gg X^{\frac{1}{2}+\epsilon}(N_1N_2\kappa^2)^{\frac{1}{2}+\epsilon}$, i.e., 
$X\gg (N_1N_2\kappa^2)^{1+\epsilon}$, 
\begin{equation}\label{e:diff_sum_Cphi_2_lower}
\sum \left|C_{\phi_1}(n)-C_{\phi_2}(n)\right|^2G\left(\frac{n}{X}\right) \gg X^{1-\epsilon}.
\end{equation}

On the other hand, by \eqref{e:diff_Cphi_2_upper}
\begin{equation}
\sum_{n=1}^\infty
\left|C_{\phi_1}(n)-C_{\phi_2}(n)\right|^2G\left(\frac{n}{X}\right)
\ll T^{-2+\epsilon} 2^k \sum_{n=1}^\infty n^{2+\epsilon}  G\left(\frac{n}{X}\right)
\ll T^{-2+\epsilon} 2^k X^{3+\epsilon}.
\end{equation}
Comparing with \eqref{e:diff_sum_Cphi_2_lower}, $X^{1-\epsilon} \ll T^{-2+\epsilon}2^k X^{3+\epsilon}$. 
For $T >2^{\frac{k}{2}}X$, i.e., any $T\gg (2^{\frac{k}{2}} N_1N_2\kappa^2)^{1+\epsilon}$, 
we get a contradiction and conclude that $\phi_1=\phi_2$. 

\thispagestyle{empty}
{\footnotesize
\nocite{*}
\bibliographystyle{amsalpha}
\bibliography{reference_fRS}
}
\end{document}